\documentclass[3p,11pt]{elsarticle}
\usepackage{amssymb}
\usepackage{amsmath}
\usepackage{amsthm}

\DeclareMathAlphabet{\mathcal}{OMS}{cmsy}{m}{n}

\makeatletter
\def\ps@pprintTitle{%
 \let\@oddhead\@empty
 \let\@evenhead\@empty
 \def\@oddfoot{\centerline{\thepage}}%
 \let\@evenfoot\@oddfoot}
\makeatother

\newcommand{\bbC}{\mathbb{C}}
\newcommand{\bbF}{\mathbb{F}}
\newcommand{\bbR}{\mathbb{R}}
\newcommand{\bbZ}{\mathbb{Z}}

\newcommand{\bfA}{\mathbf{A}}
\newcommand{\bfe}{\mathbf{e}}
\newcommand{\bfE}{\mathbf{E}}
\newcommand{\bff}{\mathbf{f}}
\newcommand{\bfF}{\mathbf{F}}
\newcommand{\bfI}{\mathbf{I}}
\newcommand{\bfJ}{\mathbf{J}}
\newcommand{\bfx}{\mathbf{x}}
\newcommand{\bfX}{\mathbf{X}}
\newcommand{\bfy}{\mathbf{y}}
\newcommand{\bfY}{\mathbf{Y}}
\newcommand{\bfZ}{\mathbf{Z}}

\newcommand{\bfone}{\boldsymbol{1}}
\newcommand{\bfzero}{\boldsymbol{0}}

\newcommand{\bfphi}{\boldsymbol{\varphi}}
\newcommand{\bfpsi}{\boldsymbol{\psi}}
\newcommand{\bfPhi}{\boldsymbol{\Phi}}
\newcommand{\bfPsi}{\boldsymbol{\Psi}}
\newcommand{\bfdelta}{\boldsymbol{\delta}}

\newcommand{\calB}{\mathcal{B}}
\newcommand{\calD}{\mathcal{D}}
\newcommand{\calG}{\mathcal{G}}
\newcommand{\calV}{\mathcal{V}}

\newcommand{\TD}{{\operatorname{TD}}}
\newcommand{\Tr}{\operatorname{Tr}}
\newcommand{\ETF}{{\operatorname{ETF}}}
\newcommand{\BIBD}{{\operatorname{BIBD}}}

\newcommand{\abs}[1]{|{#1}|}

\newcommand{\paren}[1]{({#1})}
\newcommand{\bigparen}[1]{\bigl({#1}\bigr)}

\newcommand{\bigbracket}[1]{\bigl[{#1}\bigr]}

\newcommand{\set}[1]{\{{#1}\}}
\newcommand{\bigset}[1]{\bigl\{{#1}\bigr\}}

\newcommand{\norm}[1]{\|{#1}\|}

\newcommand{\ip}[2]{\langle{#1},{#2}\rangle}

\newtheorem{theorem}{Theorem}[section]
\newtheorem{lemma}[theorem]{Lemma}
\newtheorem{corollary}[theorem]{Corollary}
\newtheorem{conjecture}[theorem]{Conjecture}
\theoremstyle{definition}
\newtheorem{definition}[theorem]{Definition}
\newtheorem{remark}[theorem]{Remark}

\begin{document}
\begin{frontmatter}
\title{Equiangular tight frames from group divisible designs}

\author[AFIT]{Matthew Fickus}
\author[SDSU]{John Jasper}

\address[AFIT]{Department of Mathematics and Statistics, Air Force Institute of Technology, Wright-Patterson AFB, OH 45433}
\address[SDSU]{Department of Mathematics and Statistics, South Dakota State University, Brookings, SD 57007}

\begin{abstract}
An equiangular tight frame (ETF) is a type of optimal packing of lines in a real or complex Hilbert space.
In the complex case, the existence of an ETF of a given size remains an open problem in many cases.
In this paper, we observe that many of the known constructions of ETFs are of one of two types.
We further provide a new method for combining a given ETF of one of these two types with an appropriate group divisible design (GDD) in order to produce a larger ETF of the same type.
By applying this method to known families of ETFs and GDDs, we obtain several new infinite families of ETFs.
The real instances of these ETFs correspond to several new infinite families of strongly regular graphs.
Our approach was inspired by a seminal paper of Davis and Jedwab which both unified and generalized McFarland and Spence difference sets.
We provide combinatorial analogs of their algebraic results,
unifying Steiner ETFs with hyperoval ETFs and Tremain ETFs.
\end{abstract}

\end{frontmatter}

\section{Introduction}

Let $N\geq D$ be positive integers, let $\bbF$ be either $\bbR$ or $\bbC$, and let $\ip{\bfx_1}{\bfx_2}=\bfx_1^*\bfx_2^{}$ be the dot product on $\bbF^D$.
The \textit{Welch bound}~\cite{Welch74} states that any $N$ nonzero vectors $\set{\bfphi_n}_{n=1}^{N}$ in $\bbF^D$ satisfy
\begin{equation}
\label{eq.Welch bound}
\max_{n\neq n'}
\tfrac{\abs{\ip{\bfphi_n}{\bfphi_{n'}}}}{\norm{\bfphi_n}\norm{\bfphi_{n'}}}
\geq\bigbracket{\tfrac{N-D}{D(N-1)}}^{\frac12}.
\end{equation}
It is well-known~\cite{StrohmerH03} that nonzero equal-norm vectors $\set{\bfphi_n}_{n=1}^{N}$ in $\bbF^D$ achieve equality in \eqref{eq.Welch bound} if and only if they form an \textit{equiangular tight frame} for $\bbF^D$,
denoted an $\ETF(D,N)$,
namely if there exists $A>0$ such that  $A\norm{\bfx}^2=\sum_{n=1}^{N}\abs{\ip{\bfphi_n}{\bfx}}^2$
for all $\bfx\in\bbF^D$ (tightness),
and the value of $\abs{\ip{\bfphi_n}{\bfphi_{n'}}}$ is constant over all $n\neq n'$ (equiangularity).
In particular, an ETF is a type of optimal packing in projective space,
corresponding to a collection of lines whose minimum pairwise angle is as large as possible.
ETFs arise in several applications,
including waveform design for communications~\cite{StrohmerH03},
compressed sensing~\cite{BajwaCM12,BandeiraFMW13},
quantum information theory~\cite{Zauner99,RenesBSC04}
and algebraic coding theory~\cite{JasperMF14}.

In the general (possibly-complex) setting,
the existence of an $\ETF(D,N)$ remains an open problem for many choices of $(D,N)$.
See~\cite{FickusM16} for a recent survey.
Beyond orthonormal bases and regular simplices,
all known infinite families of ETFs arise from combinatorial designs.
Real ETFs in particular are equivalent to a class of \textit{strongly regular graphs} (SRGs)~\cite{vanLintS66,Seidel76,HolmesP04,Waldron09},
and such graphs have been actively studied for decades~\cite{Brouwer07,Brouwer17,CorneilM91}.
This equivalence has been partially generalized to the complex setting in various ways,
including approaches that exploit properties of roots of unity~\cite{BodmannPT09,BodmannE10},
abelian distance-regular covers of complete graphs (DRACKNs)~\cite{CoutinkhoGSZ16},
and association schemes~\cite{IversonJM16}.
Conference matrices, Hadamard matrices, Paley tournaments and quadratic residues are related, and lead to infinite families of ETFs whose \textit{redundancy} $\frac ND$ is either nearly or exactly two~\cite{StrohmerH03,HolmesP04,Renes07,Strohmer08}.
\textit{Harmonic ETFs} and \textit{Steiner ETFs} offer more flexibility in choosing $D$ and $N$.
Harmonic ETFs are equivalent to \textit{difference sets} in finite abelian groups~\cite{Turyn65,StrohmerH03,XiaZG05,DingF07},
while Steiner ETFs arise from~\textit{balanced incomplete block designs} (BIBDs)~\cite{GoethalsS70,FickusMT12}.
Recent generalizations of Steiner ETFs have led to new infinite families of ETFs arising from projective planes that contain hyperovals~\cite{FickusMJ16} as well as from Steiner triple systems~\cite{FickusJMP18}, dubbed \textit{hyperoval ETFs} and \textit{Tremain ETFs}, respectively.
Another new family arises by generalizing the SRG construction of~\cite{Godsil92} to the complex setting, using generalized quadrangles to produce abelian DRACKNs~\cite{FickusJMPW19}.

Far less is known in terms of necessary conditions on the existence of complex $\ETF(D,N)$.
The \textit{Gerzon bound} implies that $N\leq\min\set{D^2,(N-D)^2}$ whenever a complex $\ETF(D,N)$ with $N>D>1$ exists~\cite{LemmensS73,HolmesP04,Tropp05}.
Beyond this, the only known nonexistence result in the complex case is that an $\ETF(3,8)$ does not exist~\cite{Szollosi14}, a result proven using computational techniques in algebraic geometry.
In quantum information theory, $\ETF(D,D^2)$ are known as \textit{symmetric informationally-complete positive operator-valued measures} (SIC-POVMs).
It is famously conjectured that such Gerzon-bound-equality ETFs exist for any $D$~\cite{Zauner99,Renes07,FuchsHS17}.

In this paper, we give a new method for constructing ETFs that yields several new infinite families of them.
Our main result is Theorem~\ref{thm.new ETF},
which shows how to combine a given initial ETF with a \textit{group divisible design} (GDD) in order to produce another ETF.
In that result, we require the initial ETF to be of one of the following types:
\begin{definition}
\label{def.ETF types}
Given integers $D$ and $N$ with $1<D<N$, we say $(D,N)$ is \textit{type $(K,L,S)$} if
\begin{equation}
\label{eq.ETF param in terms of type param}
D
=\tfrac{S}{K}[S(K-1)+L]
=S^2-\tfrac{S(S-L)}{K},
\quad
N
=(S+L)[S(K-1)+L],
\end{equation}
where $K$ and $S$ are integers and $L$ is either $1$ or $-1$.
For a given $K$, we say $(D,N)$ is \textit{$K$-positive} or \textit{$K$-negative} when it is type $(K,1,S)$ or type $(K,-1,S)$ for some $S$, respectively.
We simply say $(D,N)$ is \textit{positive} or \textit{negative} when it is $K$-positive or $K$-negative for some $K$, respectively.
When we say that an ETF is one of these types, we mean its $(D,N)$ parameters are of that type.
\end{definition}

It turns out that every known $\ETF(D,N)$ with $N>2D>2$ is either a harmonic ETF, a SIC-POVM, or is positive or negative.
In particular, every Steiner ETF is positive, while every hyperoval ETF and Tremain ETF is negative.
In this sense, the ideas and results of this paper are an attempt to unify and generalize several constructions that have been regarded as disparate.
This is analogous to---and directly inspired by---a seminal paper of Davis and Jedwab~\cite{DavisJ97}, which unifies \textit{McFarland}~\cite{McFarland73} and \textit{Spence}~\cite{Spence77} difference sets under a single framework,
and also generalizes them so as to produce difference sets whose corresponding harmonic ETFs have parameters
\begin{equation}
\label{eq.Davis Jedwab parameters}
D=\tfrac13 2^{2J-1}(2^{2J+1}+1),
\quad
N=\tfrac13 2^{2J+2}(2^{2J}-1),
\end{equation}
for some $J\geq 1$.
It is quickly verified that such ETFs are type $(4,-1,S)$ where $S=\frac13(2^{2J+1}+1)$.
As we shall see, combining our main result (Theorem~\ref{thm.new ETF}) with known ETFs and GDDs recovers the existence of ETFs with these parameters, and also provides several new infinite families, including:

\begin{theorem}
\label{thm.new neg ETS with K=4,5}
An $\ETF(D,N)$ of type $(K,-1,S)$ exists whenever:
\begin{enumerate}
\renewcommand{\labelenumi}{(\alph{enumi})}
\item
$K=4$ and either $S\equiv 3\bmod 8$ or $S\equiv 7\bmod 60$;\smallskip
\item
$K=5$ and either $S\equiv 4\bmod 15$ or $S\equiv 5,309\bmod 380$ or $S\equiv 9\bmod 280$.
\end{enumerate}
\end{theorem}

This result extends the $S$ for which an ETF of type $(4,-1,S)$ is known to exist from a geometric progression to a finite union of arithmetic progressions,
with the smallest new ETF having $S=19$, namely $(D,N)=(266,1008)$, cf.\ \cite{FickusM16}.
Meanwhile, the ETFs given by Theorem~\ref{thm.new neg ETS with K=4,5} in the $K=5$ case seem to be completely new except when $S=4,5,9$,
with $(D,N)=(285,1350)$ being the smallest new example.
Using similar techniques, we were also able to find new, explicit infinite families of $K$-negative ETFs for $K=6,7,10,12$.
The description of these families is technical, and so is given in Theorem~\ref{thm.new neg ETS with K>5} as opposed to here.
More generally, using asymptotic existence results for GDDs,
we show that an infinite number of $K$-negative ETFs also exist whenever $K=Q+2$ where $Q$ is a prime power, $K=Q+1$ where $Q$ is an even prime power,
or $K=8,20,30,42,56,342$.

In certain cases, the new ETFs constructed by these methods can be chosen to be real:

\begin{theorem}\
\label{thm.new real ETFs}
\begin{enumerate}
\renewcommand{\labelenumi}{(\alph{enumi})}
\item
There are an infinite number of real Hadamard matrices of size $H\equiv 1\bmod 35$,
and a real ETF of type $(5,-1,8H+1)$ exists for all such $H$.
\item
There are an infinite number of real Hadamard matrices of size $H\equiv 1,8\bmod 21$,
and a real ETF of type $(6,-1,2H+1)$ exists for all sufficiently large such $H$.
\item
There are an infinite number of real Hadamard matrices of size $H\equiv 1,12\bmod 55$,
and a real ETF of type $(10,-1,4H+1)$ exists for all sufficiently large such $H$.
\item
There are an infinite number of real Hadamard matrices of size $H\equiv 1,277\bmod 345$,
and a real ETF of type $(15,-1,4H+1)$ exists for all sufficiently large such $H$.
\end{enumerate}
\end{theorem}
These correspond to four new infinite families of SRGs, with the smallest new example being a real $\ETF(66759,332640)$, which is obtained by letting $H=36$ in (a).

In the next section, we introduce known concepts from frame theory and combinatorial design that we need later on.
In Section~3, we provide an alternative characterization of when an ETF is positive or negative (Theorem~\ref{thm.parameter types}),
which we then use to help prove our main result (Theorem~\ref{thm.new ETF}).
In the fourth section, we discuss how many known ETFs are either positive or negative,
and then apply Theorem~\ref{thm.new ETF} to them along with known GDDs to obtain the new infinite families of negative ETFs described in Theorems~\ref{thm.new neg ETS with K=4,5} and~\ref{thm.new neg ETS with K>5}.
We conclude in Section~5, using these facts as the basis for new conjectures on the existence of real and complex ETFs.

\section{Preliminaries}

\subsection{Equiangular tight frames}

For any positive integers $N$ and $D$,
and any sequence $\set{\bfphi_n}_{n=1}^{N}$ of vectors in $\bbF^D$,
the corresponding \textit{synthesis operator} is $\bfPhi:\bbF^N\rightarrow\bbF^D$,
$\bfPhi\bfy:=\sum_{n=1}^{N}\bfy(n)\bfphi_n$,
namely the $D\times N$ matrix whose $n$th column is $\bfphi_n$.
Its adjoint (conjugate transpose) is the \textit{analysis operator} $\bfPhi^*:\bbF^D\rightarrow\bbF^N$, which has $(\bfPhi^*\bfx)(n)=\ip{\bfphi_n}{\bfx}$ for all $n=1,\dotsc,N$.
That is, $\bfPhi^*$ is the $D\times N$ matrix whose $n$th row is $\bfphi_n^*$.
Composing these two operators gives the $N\times N$ \textit{Gram matrix} $\bfPhi^*\bfPhi$ whose $(n,n')$th entry is $(\bfPhi^*\bfPhi)(n,n')=\ip{\bfphi_n}{\bfphi_{n'}}$,
as well as the $D\times D$ \textit{frame operator} $\bfPhi\bfPhi^*=\sum_{n=1}^{N}\bfphi_n^{}\bfphi_n^*$.

We say $\set{\bfphi_n}_{n=1}^{N}$ is a \textit{tight frame} for $\bbF^D$ if there exists $A>0$ such that $\bfPhi\bfPhi^*=A\bfI$,
namely if the rows of $\bfPhi$ are orthogonal and have an equal nontrivial norm.
We say $\set{\bfphi_n}_{n=1}^{N}$ is \textit{equal norm} if there exists some $C$ such that $\norm{\bfphi_n}^2=C$ for all $n$.
The parameters of an equal norm tight frame are related according to
$DA=\Tr(A\bfI)=\Tr(\bfPhi\bfPhi^*)=\Tr(\bfPhi^*\bfPhi)=\sum_{n=1}^{N}\norm{\bfphi_n}^2=NC$.
We say $\set{\bfphi_n}_{n=1}^{N}$ is \textit{equiangular} if it is equal norm and the value of $\abs{\ip{\bfphi_n}{\bfphi_{n'}}}$ is constant over all $n\neq n'$.

For any equal norm vectors $\set{\bfphi_n}_{n=1}^{N}$ in $\bbF^D$,
a direct calculation reveals
\begin{equation*}
0
\leq\Tr[(\tfrac1C\bfPhi\bfPhi^*-\tfrac{N}{D}\bfI)^2]
=\sum_{n=1}^{N}\sum_{\substack{n'=1\\n'\neq n}}^{N}
\tfrac{\abs{\ip{\bfphi_n}{\bfphi_{n'}}}^2}{C^2}
-\tfrac{N(N-D)}{D}
\leq N(N-1)\max_{n\neq n'}
\tfrac{\abs{\ip{\bfphi_n}{\bfphi_{n'}}}^2}{C^2}
-\tfrac{N(N-D)}{D}.
\end{equation*}
Rearranging this inequality gives the Welch bound~\eqref{eq.Welch bound}.
Moreover, we see that achieving equality in~\eqref{eq.Welch bound}
is equivalent to having equality above throughout, which happens precisely when $\set{\bfphi_n}_{n=1}^{N}$ is a tight frame for $\bbF^D$ that is also equiangular,
namely when it is an ETF for $\bbF^D$.

If $N>D$ and $\set{\bfphi_n}_{n=1}^{N}$ is a tight frame for $\bbF^D$ then completing the $D$ rows of $\bfPhi$ to an equal-norm orthogonal basis for $\bbF^N$ is equivalent to taking a $(N-D)\times N$ matrix $\bfPsi$ such that
$\bfPsi^*\bfPsi=A\bfI$, $\bfPhi\bfPsi^*=\bfzero$ and $\bfPhi^*\bfPhi+\bfPsi^*\bfPsi=A\bfI$.
The sequence $\set{\bfpsi_n}_{n=1}^{N}$ of columns of any such $\bfPsi$ is called a \textit{Naimark complement} of $\set{\bfphi_n}_{n=1}^{N}$.
Since $\bfPsi\bfPsi^*=A\bfI$ and $\bfPsi^*\bfPsi=A\bfI-\bfPhi^*\bfPhi$,
any Naimark complement of an $\ETF(D,N)$ is an $\ETF(N-D,N)$.
Since any nontrivial scalar multiple of an ETF is another ETF,
we will often assume without loss of generality that a given $\ETF(D,N)$ and its Naimark complements satisfy
\begin{equation}
\label{eq.ETF scaling}
A=N\bigbracket{\tfrac{N-1}{D(N-D)}}^{\frac12},
\quad
\norm{\bfphi_n}^2
=\bigbracket{\tfrac{D(N-1)}{N-D}}^{\frac12},
\quad
\norm{\bfpsi_n}^2
=\bigbracket{\tfrac{(N-D)(N-1)}D}^{\frac12},
\quad
\forall n=1,\dotsc,N,
\end{equation}
which equates to having
$\abs{\ip{\bfphi_n}{\bfphi_{n'}}}=1=\abs{\ip{\bfpsi_n}{\bfpsi_{n'}}}$ for all $n\neq n'$.
For positive and negative ETFs in particular (Definition~\ref{def.ETF types}),
we shall see that all of these quantities happen to be integers.

Any $\ETF(D,N)$ with $N=D+1$ is known as a \textit{regular simplex},
and such ETFs are Naimark complements of ETFs for $\bbF^1$,
namely sequences of scalars that have the same nontrivial modulus.
In particular, a sequence of vectors $\set{\bff_n}_{n=1}^{N}$ in $\bbF^{N-1}$ is a Naimark complement of the all-ones sequence in $\bbF^1$ if and only if
\begin{equation}
\label{eq.regular simplex}
\bfF\bfF^*=N\bfI,
\quad
\sum_{n=1}^{N}\bff_n=\bfF\bfone=\bfzero,
\quad
\bfF^*\bfF=N\bfI-\bfJ,
\end{equation}
where $\bfone$ and $\bfJ$ denote an all-ones column vector and matrix, respectively.
Equivalently, the vectors $\set{1\oplus\bff_n}_{n=1}^{N}$ in $\bbF^N$ are equal norm and orthogonal.
In particular, for any $N>1$, we can always take $\set{1\oplus\bff_n}_{n=1}^{N}$ to be the columns of a possibly-complex Hadamard matrix of size $N$.
In this case, $\bfF$ satisfies~\eqref{eq.regular simplex} and is also~\textit{flat},
meaning every one of its entries has modulus one.
As detailed below,
flat regular simplices can be used to construct several families of ETFs,
including Steiner ETFs as well as those we introduce in Theorem~\ref{thm.new ETF}.

Harmonic ETFs are the best-known class of ETFs~\cite{Strohmer08,XiaZG05,DingF07}.
A harmonic $\ETF(D,N)$ is obtained by restricting the characters of an abelian group $\calG$ of order $N$ to a \textit{difference set} of cardinality $D$,
namely a $D$-element subset $\calD$ of $\calG$ with the property that the cardinality of $\set{(d,d')\in\calD\times\calD: g=d-d'}$ is constant over all $g\in\calG$, $g\neq 0$.
The set complement $\calG\backslash\calD$ of any difference set in $\calG$ is another difference set, and the two corresponding harmonic ETFs are Naimark complements.
In particular, for any abelian group $\calG$ of order $N$,
the harmonic ETF arising from $\calG\backslash\set{0}$ is a flat regular simplex that satisfies \eqref{eq.regular simplex}.

\subsection{Group divisible designs}

For a given integer $K\geq 2$,
a $K$-GDD is a set $\calV$ of $V>K$ vertices,
along with collections $\calG$ and $\calB$ of subsets of $\calV$,
called \textit{groups} and \textit{blocks}, respectively,
with the property that the groups partition $\calV$,
every block has cardinality $K$,
and any two vertices are either contained in a common group or a common block, but not both.
A $K$-GDD is \textit{uniform} if its groups all have the same cardinality $M$,
denoted in \textit{exponential notation} as a ``$K$-GDD of type $M^U$" where $V=UM$.

Letting $B$ be the number of blocks, a $\set{0,1}$-valued $B\times UM$ incidence matrix $\bfX$ of a $K$-GDD of type $M^U$ has the property that each row of $\bfX$ contains exactly $K$ ones.
Moreover, for any $v=1,\dotsc,V=UM$,
the $v$th column of $\bfX$ is orthogonal to $M-1$ other columns of $\bfX$,
and has a dot product of $1$ with each of the remaining $(U-1)M$ columns.
This implies that the \textit{replication number} $R_v$ of blocks that contain the $v$th vertex satisfies
\begin{equation*}
(U-1)M
=\sum_{\substack{v'=1\\v'\neq v}}^V(\bfX^*\bfX)(v,v')
=\sum_{b=1}^{B}\bfX(b,v)\sum_{\substack{v'=1\\v'\neq v}}^V\bfX(b,v')
=\sum_{b=1}^{B}\left\{\begin{array}{cl}K-1,&\bfX(b,v)=1\\0,&\bfX(b,v)=0\end{array}\right\}
=R_v(K-1).
\end{equation*}
As such, this number $R_v=R$ is independent of $v$.
At this point, summing all entries of $\bfX$ gives $BK=VR$ and so
$B$ is also uniquely determined by $K$, $M$ and $U$.
Because of this, the existence of a $K$-GDD of type $M^U$ is equivalent to that of a $\set{0,1}$-valued $B\times UM$ matrix $\bfX$ with
\begin{equation}
\label{eq.GDD incidence matrix}
R=\tfrac{M(U-1)}{K-1},
\quad
B=\tfrac{MUR}{K}=\tfrac{M^2U(U-1)}{K(K-1)},
\quad
\bfX\bfone=K\bfone,
\quad
\bfX^*\bfX
=R\,\bfI+(\bfJ_U-\bfI_U)\otimes\bfJ_M.
\end{equation}
In the special case where $M=1$, a $K$-GDD of type $1^U$ is called a $\BIBD(U,K,1)$.
In the special case where $U=K$, a $K$-GDD of type $M^K$ is called a \textit{transversal design} $\TD(K,M)$,
which is equivalent to a collection of $K-2$ \textit{mutually orthogonal Latin squares} (MOLS) of size $M$.

In order for a $K$-GDD of type $M^U$ to exist,
the expressions for $R$ and $B$ given in~\eqref{eq.GDD incidence matrix} are necessarily integers.
Beyond this, we necessarily have $U\geq K$ since we can partition any given block into its intersections with the groups,
and the cardinality of these intersections is at most one.
Altogether, the parameters of a $K$-GDD of type $M^U$ necessarily satisfy
\begin{equation}
\label{eq.GDD necessary conditions}
U\geq K,
\quad
\tfrac{M(U-1)}{K-1}\in\bbZ,
\quad
\tfrac{M^2U(U-1)}{K(K-1)}\in\bbZ.
\end{equation}
Though these necessary conditions are not sufficient~\cite{Ge07},
they are asymptotically sufficient in two distinct ways:
for any fixed $K\geq 2$ and $M\geq 1$,
there exists $U_0=U_0(K,M)$ such that a $K$-GDD of type $M^U$ exists for all $U\geq U_0$ such that~\eqref{eq.GDD necessary conditions} is satisfied~\cite{Chang76,LamkenW00};
for any fixed $U\geq K\geq 2$,
there exists $M_0=M_0(K,U)$ such that a $K$-GDD of type $M^U$ exists for all $M\geq M_0$ such that~\eqref{eq.GDD necessary conditions} is satisfied~\cite{Mohacsy11}.
In the $M=1$ and $U=K$ cases,
these facts reduce to more classical asymptotic existence results for BIBDs and MOLS, respectively.

Many specific examples of GDDs are formed by combining smaller designs in clever ways.
We in particular will make use of the following result,
which is a special case of Wilson's approach~\cite{Wilson72}:

\begin{lemma}
\label{lem.Wilson}
If a $K$-GDD of type $M^U$ and a $U$-GDD of type $N^V$ exist, then a $K$-GDD of type $(MN)^V$ exists.
\end{lemma}

\begin{proof}
Let $\bfX$ and $\bfY$ be incidence matrices of the form~\eqref{eq.GDD incidence matrix} for the given $K$-GDD of type $M^U$ and $U$-GDD of type $N^V$, respectively.
In particular, taking $R$ and $B$ as in~\eqref{eq.GDD incidence matrix},
we can write $\bfX=\left[\begin{array}{ccc}\bfX_1&\cdots&\bfX_U\end{array}\right]$ where each $\bfX_u$ is a $B\times M$ matrix with $\bfX_u^*\bfX_u^{}=R\bfI$,
and $\bfX_u^*\bfX_{u'}^{}=\bfJ$ for any $u\neq u'$.
We now construct the incidence matrix $\bfZ$ of a $K$-GDD of type $(MN)^V$ in the following manner:
in each row of $\bfY$, replace each of the $U$ nonzero entries with a distinct matrix $\bfX_u$,
and replace each of the zero entries with a $B\times M$ matrix of zeros.
\end{proof}

This result generalizes MacNeish's classical method for combining MOLS~\cite{MacNeish22}:
if a $\TD(K,M)$ and a $\TD(K,N)$ exist, then applying Lemma~\ref{lem.Wilson} to them produces a $\TD(K,MN)$.
We will also use one GDD to ``fill the holes" of another:

\begin{lemma}
\label{lem.filling holes}
If $K$-GDDs of type $M^U$ and $(MU)^V$ exist, then a $K$-GDD of type $M^{UV}$ exists.
\end{lemma}

\begin{proof}
Letting $\bfX$ and $\bfY$ be incidence matrices of the form~\eqref{eq.GDD incidence matrix} for the given $K$-GDDs of type $M^U$ and $(MU)^V$, respectively,
it is straightforward to verify that
\begin{equation*}
\bfZ
=\left[\begin{array}{c}\bfI_V\otimes\bfX\\\bfY\end{array}\right]
\end{equation*}
is the incidence matrix of a $K$-GDD of type $M^{UV}$.
\end{proof}

\subsection{Previously known constructions of ETFs involving BIBDs and MOLS}

In the next section, we introduce a method for constructing ETFs that uses GDDs.
This method makes use of a concept from~\cite{FickusJMP18}, which we now generalize from BIBDs to GDDs:

\begin{definition}
\label{def.embeddings}
Take a $K$-GDD of type $M^U$ where $M\geq1$ and $U\geq K\geq 2$,
and define $R$, $B$ and an incidence matrix $\bfX$ according to~\eqref{eq.GDD incidence matrix}.
Without loss of generality, write the columns of $\bfX$ as \smash{$\set{\bfx_{u,m}}_{u=1,}^{U}\,_{m=1}^{M}$} where, for each $u$,
the vectors $\set{\bfx_{u,m}}_{m=1}^{M}$ have disjoint support.
Then, for any $u$ and $m$, a corresponding \textit{embedding operator} $\bfE_{u,m}$ is any $\set{0,1}$-valued $B\times R$ matrix whose columns are standard basis elements that sum to $\bfx_{u,m}$.
\end{definition}

In the special case where $M=1$,
this concept leads to an elegant formulation of Steiner ETFs~\cite{FickusJMP18}:
letting \smash{$\set{\bfE_v}_{v=1}^{V}=\set{\bfE_{u,1}}_{u=1}^{U}$} be the embedding operators of a $\BIBD(V,K,1)$,
and letting $\set{1\oplus\bff_i}_{i=0}^{R}$ be the columns of a possibly-complex Hadamard matrix of size $R+1=\frac{V-1}{K-1}+1$,
the $V(R+1)$ vectors \smash{$\set{\bfE_v\bff_i}_{v=1,}^{V}\,_{i=0}^{R}$} form an ETF for $\bbF^B$.
In~\cite{FickusJMP18}, this fact is proven using several properties of embedding operators.
We now show those properties generalize to the GDD setting;
later on, we use these facts to prove our main result:

\begin{lemma}
\label{lem.embed}
If $\set{\bfE_{u,m}}_{u=1,}^{U},\,_{m=1}^{M}$ are the embedding operators arising from a $K$-GDD of type $M^U$,
\begin{equation*}
\bfE_{u,m}^*\bfE_{u',m'}^{}
=\left\{\begin{array}{cl}
\bfI,&\ u=u', m=m',\\
\bfzero,&\ u=u', m\neq m',\\
\bfdelta_{r}^{}\bfdelta_{\smash{r'}}^*,&\ u\neq u'.
\end{array}\right.
\end{equation*}
Here, for any $u\neq u'$ and $m,m'$,
$\bfdelta_r$ and $\bfdelta_{\smash{r'}}$ are standard basis elements in $\bbF^R$ whose indices $r,r'$ depend on $u,u',m,m'$.
\end{lemma}

\begin{proof}
Each $\bfE_{u,m}$ is a matrix whose columns are standard basis elements that sum to~$\bfx_{u,m}$, and so is an isometry, that is, $\bfE_{u,m}^*\bfE_{u,m}^{}=\bfI$.
Moreover, for any $u$, $u'$, $m$, $m'$, $\bfE_{u,m}^*\bfE_{u',m'}^{}$ is a matrix whose entries are nonnegative integers that sum to:
\begin{equation*}
\sum_{r=1}^{R}\sum_{r'=1}^{R}(\bfE_{u,m}^*\bfE_{u',m'}^{})(r,r')
=\bfone^*\bfE_{u,m}^*\bfE_{u',m'}^{}\bfone^{}
=\ip{\bfx_{u,m}}{\bfx_{u',m'}}
=\left\{\begin{array}{cl}
R,&\ u=u', m=m',\\
0,&\ u=u', m\neq m',\\
1,&\ u'\neq u.
\end{array}\right.
\end{equation*}
When $u=u'$ and $m\neq m'$, this implies $\bfE_{u,m}^*\bfE_{u',m'}^{}=\bfzero$.
If instead $u\neq u'$ then this implies that $\bfE_{u,m}^*\bfE_{u',m'}^{}$ has a single nonzero entry, and that this entry has value $1$.
This means there exists some $r,r'=1,\dotsc,R$, $r=r(u,m,u',m')$, $r'=r'(u,m,u',m')$ such that $\bfE_{u,m}^*\bfE_{\smash{u,m'}}^{}=\bfdelta_{r}^{}\bfdelta_{\smash{r'}}^*$.
\end{proof}

Other Steiner-like constructions of ETFs include hyperoval ETFs~\cite{FickusMJ16} and Tremain ETFs~\cite{FickusJMPW19}.
Beyond Steiner and Steiner-like techniques,
there are at least two other methods for constructing ETFs that make direct use of the incidence matrix of some kind of GDD.
One method leads to the \textit{phased BIBD ETFs} of~\cite{FickusJMPW19}:
if $\bfX$ is the $B\times V$ incidence matrix of a $\BIBD(V,K,1)$,
and $\bfPhi$ is any matrix obtained by replacing each $1$-valued entry of $\bfX$ with any unimodular scalar,
then the columns of $\bfPhi$ are immediately equiangular,
and the challenge is to design them so that they form a tight frame for their span.
Another method constructs ETFs with
$(D,N)=(\tfrac12M(M\pm1),M^2)$ from MOLS.
To elaborate, a $\TD(K,M)$ is a $K$-GDD of type $M^K$,
meaning by~\eqref{eq.GDD incidence matrix} that it has an $M^2\times KM$ incidence matrix $\bfX$ that satisfies
\begin{equation}
\label{eq.incidence matrix of TD 1}
\bfX\bfone=K\bfone,
\quad
\bfX^*\bfX=M\bfI+(\bfJ_K-\bfI_K)\otimes\bfJ_M.
\end{equation}
Here, the columns of $\bfX$ have support $M$,
and are arranged as $K$ groups of $M$ columns apiece,
where columns in a common group have disjoint support.
Together, these facts imply, in turn, that
\begin{equation}
\label{eq.incidence matrix of TD 2}
\bfX(\bfI_K\otimes\bfone_M)=\bfone_{M^2}^{}\bfone_K^*,
\quad
(\bfX\bfX^*)^2=M\bfX\bfX^*+K(K-1)\bfJ.
\end{equation}
At this point, the traditional approach is to let $\bfA=\bfX\bfX^*-K\bfI$ be the adjacency matrix of the TD's \textit{block graph},
and use~\eqref{eq.incidence matrix of TD 1} and \eqref{eq.incidence matrix of TD 2} to show that this graph is strongly regular with parameters $(M^2,K(M-1),M+K(K-3),K(K-1))$.
In the $M=2K$ case, applying Theorem~4.4 of~\cite{FickusJMPW18} to this graph then produces a real ETF with $(D,N)=(\frac12M(M-1),M^2)$ whose vectors sum to zero,
while applying this same result in the $M=2(K-1)$ case produces a real ETF with $(D,N)=(\frac12M(M+1),M^2)$ whose synthesis operator's row space contains the all-ones vector.

That said, a careful read of the literature reveals that this construction can be made more explicit,
and that doing so has repercussions for coding theory.
To elaborate, in~\cite{BrackenMW06}, MOLS are used to produce quasi-symmetric designs (QSDs) which, via the techniques of~\cite{McGuire97}, yield self-complementary binary codes that achieve equality in the Grey-Rankin bound.
In~\cite{JasperMF14}, such codes are shown to be equivalent to flat real ETFs.
A method for directly converting the incidence matrices of certain QSDs into synthesis operators of ETFs was also recently introduced~\cite{FickusJMP19}.
Distilling these ideas leads to the following streamlined construction:
let $\bfX$ be the incidence matrix of a $\TD(K,M)$,
let $\set{1\oplus\bff_m}_{m=1}^{M}$ be the columns of a possibly-complex Hadamard matrix of size $M$,
let $\bfF$ be the $(M-1)\times M$ synthesis operator of $\set{\bff_m}_{m=1}^{M}$,
and consider the $K(M-1)\times M^2$ matrix
\begin{equation}
\label{eq.first flat ETF from MOLS}
\bfPhi=(\bfI_K\otimes\bfF)\bfX^*.
\end{equation}
Using~\eqref{eq.regular simplex},
\eqref{eq.incidence matrix of TD 1},
and~\eqref{eq.incidence matrix of TD 2} along with the fact that $\bfF$ is flat,
it is straightforward to show that $\bfPhi$ is flat and satisfies
$\bfPhi\bfPhi^*
=M^2\bfI$
and
$\bfPhi^*\bfPhi
=M\bfX\bfX^*-K\bfJ$.
As such, the columns of $\bfPhi$ form a flat \textit{two-distance tight frame} (TDTF) for $\bbF^{K(M-1)}$~\cite{BargGOY15}.
Moreover, this TDTF is an ETF when $M=2K$.
In particular, if there exists a $\TD(K,2K)$ and a real Hadamard matrix of size $2K$,
then there exists a flat real $\ETF(K(2K-1),4K^2)$.
Using the equivalence between flat real ETFs and Grey-Rankin-bound-equality codes given in~\cite{JasperMF14}, or alternatively the equivalence between such ETFs and certain QSDs given in~\cite{FickusJMP19}, this recovers Theorem~1 of~\cite{BrackenMW06}.
In the $K=6$ case, that result gives the only known proof to date of the existence of a flat real $\ETF(66,144)$.

For any TD and corresponding flat regular simplex,
it is quickly verified that the corresponding TDTF~\eqref{eq.first flat ETF from MOLS} is \textit{centered}~\cite{FickusJMPW18} in the sense that
$\bfPhi\bfone
=\bfzero$,
namely that the all-ones vector is orthogonal to the row space of $\bfPhi$.
This fact leads to an analogous reinterpretation of the second main result of~\cite{BrackenMW06}:
in lieu of~\eqref{eq.first flat ETF from MOLS}, we instead consider the $[K(M-1)+1]\times M^2$ flat matrix
\begin{equation}
\label{eq.second flat ETF from MOLS}
\bfPsi
=\left[\begin{array}{l}\bfone^*\\\bfPhi\end{array}\right]
=\left[\begin{array}{c}\bfone^*\\(\bfI_K\otimes\bfF)\bfX^*\end{array}\right].
\end{equation}
Here, the properties of $\bfPhi$ immediately imply $\bfPsi\bfPsi^*=M^2\bfI$ and $\bfPsi^*\bfPsi=M\bfX\bfX^*-(K-1)\bfJ$,
meaning the columns of $\bfPsi$ form a flat TDTF.
However, unlike~\eqref{eq.first flat ETF from MOLS},
the columns of~\eqref{eq.second flat ETF from MOLS} are equiangular precisely when $M=2(K-1)$.
Replacing $K$ with $K+1$, this implies in particular that if there exists a $\TD(K+1,2K)$ and a real Hadamard matrix of size $2K$,
then there exists a flat real $\ETF(K(2K+1),4K^2)$.
This recovers Theorem~2 of~\cite{BrackenMW06} via the equivalences of~\cite{JasperMF14,FickusJMP19},
and gives the only known proof of the existence of a flat real $\ETF(78,144)$.

Simply put, if certain TDs exist, then certain ETFs exist.
In the next section, we introduce a new method of constructing ETFs from TDs.

\section{Constructing equiangular tight frames with group divisible designs}

In~\cite{DavisJ97}, Davis and Jedwab unify McFarland~\cite{McFarland73} and Spence~\cite{Spence77} difference sets under a single framework, and also generalize them so as to produce difference sets with parameters~\eqref{eq.Davis Jedwab parameters}.
McFarland's construction relies on nice algebro-combinatorial properties of the set of all hyperplanes in a finite-dimensional vector space over a finite field.
Davis and Jedwab exploit these properties to form various types of \textit{building sets}, which in some cases lead to difference sets.

In~\cite{JasperMF14}, it is shown that every harmonic ETF arising from a McFarland ETF is unitarily-equivalent to a Steiner ETF arising from an affine geometry.
When we applied a similar analysis to the building sets of~\cite{DavisJ97}, we discovered that they have an underlying TD-like incidence structure.
(We do not provide this analysis here since it is nontrivial and does not help us prove our results in their full generality.)
This eventually led us to the ETF construction technique of Theorem~\ref{thm.new ETF} below.
In short, our approach here is directly inspired by that of~\cite{DavisJ97}, though this is not apparent from our proof techniques.
In particular, the fact that the $L$ parameter in Definition~\ref{def.ETF types} is either $1$ or $-1$ is a generalization of Davis and Jedwab's notion of \textit{extended building sets} with ``$+$" and ``$-$" parameters, respectively.
To facilitate our arguments later on,
we now consider these types of parameters in greater detail:

\begin{theorem}
\label{thm.parameter types}
If $1<D<N$ and $(D,N)$ is type $(K,L,S)$, see Definition~\ref{def.ETF types}, then
\begin{equation}
\label{eq.type param in terms of ETF param}
S=\bigbracket{\tfrac{D(N-1)}{N-D}}^{\frac12},
\quad
K=\tfrac{NS}{D(S+L)},
\end{equation}
where $S\geq 2$.
Conversely, given $(D,N)$ such that $1<D<N$, and letting $L$ be either $1$ or $-1$,
if the above expressions for $S$ and $K$ are integers then $(D,N)$ is type $(K,L,S)$.

Moreover, in the case that the equivalent conditions above hold, scaling an ETF $\set{\bfphi_n}_{n=1}^{N}$ for $\bbF^D$ so that $\abs{\ip{\bfphi_n}{\bfphi_{n'}}}=1$ for all $n\neq n'$ gives that it and its Naimark complements $\set{\bfpsi_n}_{n=1}^{N}$ have tight frame constant $A=K(S+L)$ and
\begin{equation}
\label{eq.norms of pos or neg ETFs}
\norm{\bfphi_n}^2
=\bigbracket{\tfrac{D(N-1)}{N-D}}^{\frac12}
=S,
\quad
\norm{\bfpsi_n}^2
=\bigbracket{\tfrac{(N-D)(N-1)}{D}}^{\frac12}
=S(K-1)+KL,
\quad
\forall n=1,\dotsc,N.
\end{equation}
\end{theorem}

\begin{proof}
Whenever $(D,N)$ is type $(K,L,S)$ we have that $L$ is either $1$ or $-1$ by assumption, at which point the fact that $L^2=1$ coupled with~\eqref{eq.ETF param in terms of type param} gives
\begin{align*}
\nonumber
N-1
&=(S+L)[S(K-1)+L]-1
=S^2(K-1)+SKL
=S[S(K-1)+KL],\\
\tfrac{N}{D}-1
&=\tfrac{K(S+L)[S(K-1)+L]}{S[S(K-1)+L]}-1
=\tfrac{K(S+L)}{S}-1
=\tfrac1{S}[S(K-1)+KL].
\end{align*}
Multiplying and dividing these expressions immediately implies that
\begin{equation}
\label{eq.N in terms of D and K 0}
\bigbracket{\tfrac{D(N-1)}{N-D}}^{\frac12}=S,
\quad
\bigbracket{\tfrac{(N-D)(N-1)}{D}}^{\frac12}=S(K-1)+KL.
\end{equation}
Here, $S$ is an integer by assumption, and is clearly positive.
Moreover, if $S=1$ then~\eqref{eq.N in terms of D and K 0} implies $D=1$.
Since $D>1$ by assumption, we thus have $S\geq 2$.
Continuing, \eqref{eq.ETF param in terms of type param} further implies
\begin{equation*}
\tfrac{NS}{D(S+L)}
=\tfrac{K(S+L)[S(K-1)+L]}{S[S(K-1)+L]}\tfrac{S}{S+L}
=K.
\end{equation*}

Conversely, now assume that $S$ and $K$ are defined by~\eqref{eq.type param in terms of ETF param}, where $L$ is either $1$ or $-1$, and that $S$ and $K$ are integers.
As before, the fact that $D>1$ implies that $S\geq 2$ and so $K>0$.
We solve for $N$ in terms of $D$, $K$, and $L$.
Here, \eqref{eq.type param in terms of ETF param} gives $\frac{N}{DK}=\frac{S+L}{S}=1+\frac{L}{S}$.
Since $L^2=1$, this implies
\begin{equation}
\label{eq.N in terms of D and K 1}
\bigbracket{\tfrac{N-D}{D(N-1)}}^{\frac12}
=\tfrac1{S}
=L\bigparen{\tfrac{N}{DK}-1}.
\end{equation}
Squaring this equation and multiplying the result by $N-1$ gives
\begin{equation*}
\tfrac{N}{D}-1
=\tfrac{N-D}{D}
=(N-1)\bigparen{\tfrac{N}{DK}-1}^2
=N\bigparen{\tfrac{N}{DK}-1}^2-\tfrac{N}{DK}\bigparen{\tfrac{N}{DK}-2}-1.
\end{equation*}
Adding $1$ to this equation and multiplying by $\tfrac{(DK)^2}{N}$ then leads to a quadratic in $N$:
\begin{equation*}
DK^2
=(N-DK)^2-(N-2DK)
=N^2-(2DK+1)N+DK(DK+2).
\end{equation*}
Applying the quadratic formula then gives
\begin{equation}
\label{eq.N in terms of D and K 2}
N=DK+\tfrac12\bigset{1\pm\bigbracket{4DK(K-1)+1}^{\frac12}}.
\end{equation}
As such, \eqref{eq.N in terms of D and K 1} becomes
$\tfrac1{S}
=L\bigparen{\tfrac{N}{DK}-1}
=\tfrac{L}{2DK}\bigset{1\pm\bigbracket{4DK(K-1)+1}^{\frac12}}$,
implying ``$+$" and ``$-$" here correspond to $L=1$ and $L=-1$ respectively, that is,
\begin{equation}
\label{eq.N in terms of D and K 3}
\tfrac1{S}
=\tfrac{L}{2DK}\bigset{1+L\bigbracket{4DK(K-1)+1}^{\frac12}}
=\tfrac1{2DK}\bigset{L+\bigbracket{4DK(K-1)+1}^{\frac12}}.
\end{equation}
Moreover, since $N$ is an integer,
\eqref{eq.N in terms of D and K 2} implies that $4DK(K-1)+1$ is the square of an odd integer,
that is, that $4DK(K-1)+1=(2J-1)^2=4J(J-1)+1$ or equivalently that $DK(K-1)=J(J-1)$ for some positive integer $J$.
Writing $D=\frac{J(J-1)}{K(K-1)}$,~\eqref{eq.N in terms of D and K 2} and \eqref{eq.N in terms of D and K 3} then become
\begin{align}
\label{eq.N in terms of D and K 4}
N
&=DK+\tfrac12\bigset{1+L\bigbracket{4DK(K-1)+1}^{\frac12}}
=DK+\tfrac12[1+L(2J-1)],
\\
\nonumber
\tfrac1{S}
&=\tfrac1{2DK}\bigset{L+\bigbracket{4DK(K-1)+1}^{\frac12}}
=\tfrac{K-1}{2J(J-1)}[L+(2J-1)]
=\left\{\begin{array}{cl}
\tfrac{K-1}{J-1},&L=1\smallskip\\
\tfrac{K-1}{J},&L=-1
\end{array}\right\}
=\tfrac{2(K-1)}{2J-L-1}.
\end{align}
That is, $J=S(K-1)+\frac12(L+1)$.
Substituting this into $D=\frac{J(J-1)}{K(K-1)}$ and~\eqref{eq.N in terms of D and K 4} and again using the fact that $L^2=1$ then gives the expressions for $D$ and $N$ given in Definition~\ref{def.ETF types}:
\begin{align*}
D
&=\tfrac{[S(K-1)+\frac12(L+1)][S(K-1)+\frac12(L-1)]}{K(K-1))}
=\tfrac{S^2(K-1)^2+S(K-1)L}{K(K-1))}
=\tfrac{S}{K}[S(K-1)+L],\\
N
&=S[S(K-1)+L]+\tfrac12\{1+L[2S(K-1)+L]\}
=(S+L)[S(K-1)+L].
\end{align*}

Finally, in the case where $(D,N)$ is type $(K,L,S)$,
if $\set{\bfphi_n}_{n=1}^{N}$ is an ETF for $\bbF^D$,
and is without loss of generality scaled so that $\abs{\ip{\bfphi_n}{\bfphi_{n'}}}=1$ for all $n\neq n'$,
then~\eqref{eq.ETF scaling} and~\eqref{eq.N in terms of D and K 0} immediately imply that $\set{\bfphi_n}_{n=1}^{N}$ and any one of its Naimark complements $\set{\bfpsi_n}_{n=1}^{N}$ satisfy~\eqref{eq.norms of pos or neg ETFs},
and that both are tight frames with tight frame constant $A=\frac{NS}{D}=K(S+L)$.
\end{proof}

Theorem~\ref{thm.parameter types} implies that the $(D,N)$ parameters of an ETF with $N>D>1$ are type $(K,L,S)$ with $K=1$ if and only if that ETF is a regular simplex, and moreover that this only occurs when $L=1$ and $S=D$.
Indeed, for any $D>1$, the pair $(D,N)=(D,D+1)$ satisfies~\eqref{eq.ETF param in terms of type param} when $(K,L,S)=(1,1,D)$.
Conversely, in light of~\eqref{eq.norms of pos or neg ETFs},
an ETF of type $(1,L,S)$ has a Naimark complement $\set{\bfpsi_n}_{n=1}^{N}$ with the property that $\norm{\bfpsi_n}^2=L=1$ and $\abs{\ip{\bfpsi_n}{\bfpsi_{n'}}}=1$ for all $n\neq n'$, namely a Naimark complement that is an $\ETF(1,N)$.

We also emphasize that it is sometimes possible for the parameters of a single ETF to be simultaneously positive and negative for different choices of $K$.
In particular, if $D>1$ and $(D,D+1)$ is type $(K,-1,S)$, then \eqref{eq.type param in terms of ETF param} gives $S=D$ and $K=\frac{NS}{D(S+L)}=\frac{(D+1)D}{D(D-1)}=\frac{D+1}{D-1}=1+\frac{2}{D-1}$.
Since $K$ is an integer, this implies either $D=2$ or $D=3$.
And, letting $(K,L,S)$ be $(3,-1,2)$ and $(2,-1,3)$ in~\eqref{eq.ETF param in terms of type param} indeed gives that $(D,N)$ is $(2,3)$ and $(3,4)$, respectively.

In the next section, we provide a much more thorough discussion of positive and negative ETFs,
including some other examples of ETFs that are both.
For now, we turn to our main result, which shows how to combine a given $\ETF(D,N)$ whose parameters are type $(K,L,S)$ with a certain $K$-GDD to produce a new ETF whose parameters are type $(K,L,S')$ for some $S'>S$.
Here, as with any GDD, we require $K\geq 2$.
In light of the above discussion,
this is not a significant restriction since $(D,N)$ has type $(1,L,S)$ if and only if $L=1$ and $S=D$,
and we already know that ETFs of type $(1,1,D)$ exist for all $D>1$, being regular simplices.

\begin{theorem}
\label{thm.new ETF}
Assume an ETF of type $(K,L,S)$ exists where $K\geq 2$,
and let $M=S(K-1)+L$.
The necessary conditions~\eqref{eq.GDD necessary conditions} on the existence of a $K$-GDD of type $M^U$ reduce to having
\begin{equation}
\label{eq.new ETF 1}
U\geq K,
\quad
\tfrac{U-1}{K-1}\in\bbZ,
\quad
\tfrac{(S-L)U(U-1)}{K(K-1)}\in\bbZ.
\end{equation}
Moreover, if such a GDD exists,
and $U$ has the additional property that
\begin{equation}
\label{eq.new ETF 2}
\tfrac{(K-2)(U-1)}{(S+L)(K-1)}\in\bbZ,
\end{equation}
then there exists an ETF of type $(K,L,S')$ where $S'=S+R=\tfrac{MU-L}{K-1}$,
where $R=\tfrac{M(U-1)}{K-1}$.\smallskip

In particular, under these hypotheses, \smash{$W:=\tfrac{R}{S+L}\in\bbZ$},
and without loss of generality writing the given ETF as \smash{$\set{\bfphi_{m,i}}_{m=1,}^{M}\,_{i=1}^{S+L}$} where $\norm{\bfphi_{m,i}}^2=S$ for all $m$ and $i$,
letting $\set{\bfE_{u,m}}_{u=1,}^{U}\,_{m=1}^{M}$ be the embedding operators of the GDD (Definition~\ref{def.embeddings}),
letting $\set{\bfdelta_u}_{u=1}^{U}$ be the standard basis for $\bbF^U$,
and letting $\set{\bfe_i}_{i=1}^{S+L}$ and $\set{1\oplus\bff_j}_{j=0}^{W}$ be the columns of possibly-complex Hadamard matrices of size $S+L$ and $W+1$, respectively,
then the following vectors form an ETF of type $(K,L,S')$:
\begin{equation}
\label{eq.new ETF 3}
\set{\bfpsi_{u,m,i,j}}_{u=1,}^{U}\,_{m=1,}^{M}\,_{i=1,}^{S+L}\,_{j=0}^{W},
\quad
\bfpsi_{u,m,i,j}:=(\bfdelta_u\otimes\bfphi_{m,i})\oplus\bigparen{\bfE_{u,m}(\bfe_i\otimes\bff_j)}.
\end{equation}

As special cases of this fact, an ETF of type $(K,L,S')$ exists whenever either:
\begin{enumerate}
\renewcommand{\labelenumi}{(\alph{enumi})}
\item
$U$ is sufficiently large and satisfies \eqref{eq.new ETF 1} and \eqref{eq.new ETF 2};
\item
there exists a $K$-GDD of type $M^U$, provided we also have that $S+L$ divides $K-2$.
\end{enumerate}
\end{theorem}

\begin{proof}
Since $M=S(K-1)+L$,
\begin{equation}
\label{eq.pf of new neg ETF 1}
\tfrac{M}{K-1}
=S+\tfrac{L}{K-1}
=(S+L)-L\bigparen{\tfrac{K-2}{K-1}},
\quad
\tfrac{M}{K}
=S-\tfrac{S-L}{K}.
\end{equation}
As such, the replication number of any $K$-GDD of type $M^U$ is $R=\frac{M(U-1)}{K-1}=S(U-1)+L(\frac{U-1}{K-1})$.
In particular, such a GDD can only exist when $K-1$ necessarily divides $U-1$.
Moreover, multiplying the expressions in~\eqref{eq.pf of new neg ETF 1} gives that the number of blocks in any such GDD is
\begin{align*}
B
&=\tfrac{M^2U(U-1)}{K(K-1)}\\
&=\bigparen{S-\tfrac{S-L}{K}}\bigparen{S+\tfrac{L}{K-1}}U(U-1)\\
&=S^2U(U-1)+LSU\tfrac{U-1}{K-1}-\tfrac{S(S-L)}{K}U(U-1)-L\tfrac{(S-L)U(U-1)}{K(K-1)}.
\end{align*}
Here, since our initial $\ETF(D,N)$ is type $(K,L,S)$, $\frac{S(S-L)}{K}=S^2-D$ is an integer, and so the above expression for $B$ is an integer precisely when $K(K-1)$ divides $(S-L)U(U-1)$.
To summarize, since $M=S(K-1)+L$ where $K$ divides $S(S-L)$,
the necessary conditions~\eqref{eq.GDD necessary conditions} on the existence of a $K$-GDD of type $M^U$ reduce to having~\eqref{eq.new ETF 1}.
These necessary conditions are known to be asymptotically sufficient~\cite{Chang76,LamkenW00}:
for this fixed $K$ and $M$, there exists $U_0$ such that a $K$-GDD of type $M^U$ exists for any $U\geq U_0$ that satisfies~\eqref{eq.new ETF 1}.
Regardless, to apply our construction below with any given $K$-GDD of type $M^U$,
we only need $U$ to satisfy the additional property that
\begin{equation*}
W
=\tfrac{R}{S+L}
=\tfrac{M(U-1)}{(S+L)(K-1)}
=\tfrac{U-1}{S+L}\bigbracket{(S+L)-L\bigparen{\tfrac{K-2}{K-1}}}
=(U-1)-L\tfrac{(K-2)(U-1)}{(S+L)(K-1)}
\end{equation*}
is an integer,
namely to satisfy~\eqref{eq.new ETF 2}.
Since $K-1$ necessarily divides $U-1$,
this is automatically satisfied whenever $S+L$ happens to divide $K-2$,
and some of the ETFs we will identify in the next section will have this nice property.
Regardless, there are always an infinite number of values of $U$ which satisfy~\eqref{eq.new ETF 1} and~\eqref{eq.new ETF 2},
including, for example, all $U\equiv 1\bmod (S+L)K(K-1)$.

Turning to the construction itself,
the fact that the given $\ETF(D,N)$ is type $(K,L,S)$ implies
\begin{equation}
\label{eq.proof of new ETF 1}
D
=\tfrac{S}{K}[S(K-1)+L]
=\tfrac{SM}{K},
\quad
N
=(S+L)[S(K-1)+L]
=(S+L)M.
\end{equation}
In particular, since $N=(S+L)M$, the vectors in our initial $\ETF(D,N)$ can indeed be indexed as \smash{$\set{\bfphi_{m,i}}_{m=1,}^{M}\,_{i=1}^{S+L}$}.
Moreover,
\begin{equation}
\label{eq.proof of new ETF 2}
MU
=M(U-1)+M
=(K-1)R+[S(K-1)+L]
=(S+R)(K-1)+L.
\end{equation}
As such, the number of vectors in the collection~\eqref{eq.new ETF 3} is
\begin{equation}
\label{eq.proof of new ETF 3}
N'
=UM(S+L)(W+1)
=(S+L)(\tfrac{R}{S+L}+1)MU
=[(S+R)+L][(S+R)(K-1)+L].
\end{equation}
Also, for each $i$ and $j$,
$\bfe_i\otimes\bff_j$ lies in a space of dimension $\bbF^{(S+L)W}=\bbF^R$.
And, for each $u$ and $m$, $\bfE_{u,m}$ is a $B\times R$ matrix.
As such, for any $u$, $m$, $i$ and $j$, $\bfpsi_{u,m,i,j}$ is a well-defined vector in $\bbF^{D'}$ where, by combining~\eqref{eq.GDD incidence matrix},
\eqref{eq.proof of new ETF 1} and~\eqref{eq.proof of new ETF 2}, we have
\begin{equation}
\label{eq.proof of new ETF 4}
D'
=UD+B
=U\tfrac{SM}{K}+\tfrac{MU}{K}R
=\tfrac{S+R}{K}MU
=\tfrac{S+R}{K}[(S+R)(K-1)+L].
\end{equation}
Comparing~\eqref{eq.proof of new ETF 3} and~\eqref{eq.proof of new ETF 4} against~\eqref{eq.ETF param in terms of type param}, we see that $(D',N')$ is indeed type $(K,L,S')$ where $S'=S+R$.
Here, \eqref{eq.proof of new ETF 2} further implies $S'=S+R=\tfrac{MU-L}{K-1}$.

Continuing, since $(D',N')$ is type $(K,L,S')$,
Theorem~\ref{thm.parameter types} gives that the Welch bound for $N'$ vectors in $\bbF^{D'}$ is \smash{$\frac1{S'}$}.
As such, to show \eqref{eq.new ETF 3} is an ETF for $\bbF^{D'}$, it suffices to prove that $\norm{\bfpsi_{u,m,i,j}}^2=S'$ for all $u$, $m$, $i$, $j$,
and that $\ip{\bfpsi_{u,m,i,j}}{\bfpsi_{u',m',i',j'}}$ is unimodular whenever $(u,m,i,j)\neq(u',m',i',j')$.
Here, for any $u,u'=1,\dotsc,U$, $m,m'=1,\dotsc,M$, $i,i'=1,\dotsc,S+L$, $j,j'=0,\dotsc,W$,
\begin{align}
\ip{\bfpsi_{u,m,i,j}}{\bfpsi_{u',m',i',j'}}
\nonumber
&=\ip{(\bfdelta_u\otimes\bfphi_{m,i})\oplus\bigparen{\bfE_{u,m}(\bfe_i\otimes\bff_j)}}{(\bfdelta_{u'}\otimes\bfphi_{m',i'})\oplus\bigparen{\bfE_{u',m'}(\bfe_{i'}\otimes\bff_{j'})}}\\
\nonumber
&=\ip{\bfdelta_u}{\bfdelta_{u'}}\ip{\bfphi_{m,i}}{\bfphi_{m',i'}}
+\ip{\bfE_{u,m}(\bfe_i\otimes\bff_j)}{\bfE_{u',m'}(\bfe_{i'}\otimes\bff_{j'})}\\
\label{eq.proof of new ETF 5}
&=\ip{\bfdelta_u}{\bfdelta_{u'}}\ip{\bfphi_{m,i}}{\bfphi_{m',i'}}
+\ip{\bfe_i\otimes\bff_j}{\bfE_{u,m}^*\bfE_{u',m'}(\bfe_{i'}\otimes\bff_{j'})}.
\end{align}
When $u=u'$, $m=m'$, $i=i'$ and $j=j'$,~\eqref{eq.proof of new ETF 5} indeed becomes
\begin{equation*}
\norm{\bfpsi_{u,m,i,j}}^2
=\norm{\bfdelta_u}^2\norm{\bfphi_{m,i}}^2+\norm{\bfe_i}^2\norm{\bff_j}^2
=1S+(S+L)W
=S+R
=S'.
\end{equation*}
As such, all that remains is to show that~\eqref{eq.proof of new ETF 5} is unimodular in all other cases.
For instance, if $u\neq u'$ then Lemma~\ref{lem.embed} gives that for any $m,m'$ there exists $r,r'=1,\dotsc,R$ such that
$\bfE_{u,m}^*\bfE_{u',m'}^{}=\bfdelta_{r}^{}\bfdelta_{\smash{r'}}^*$ meaning in this case~\eqref{eq.proof of new ETF 5} becomes
\begin{equation*}
\ip{\bfpsi_{u,m,i,j}}{\bfpsi_{u',m',i',j'}}
=0\ip{\bfphi_{m,i}}{\bfphi_{m',i'}}
+\ip{\bfe_i\otimes\bff_j}{\bfdelta_{r}^{}\bfdelta_{\smash{r'}}^*(\bfe_{i'}\otimes\bff_{j'})}
=\overline{(\bfe_i\otimes\bff_j)(r)}(\bfe_{i'}\otimes\bff_{j'})(r'),
\end{equation*}
which is unimodular, being a product of unimodular numbers.
If we instead have $u=u'$ and $m\neq m'$ then Lemma~\ref{lem.embed} gives $\bfE_{u,m}^*\bfE_{u,m'}^{}=\bfzero$ and so~\eqref{eq.proof of new ETF 5} becomes
\begin{equation*}
\ip{\bfpsi_{u,m,i,j}}{\bfpsi_{u,m',i',j'}}
=\norm{\bfdelta_u}^2\ip{\bfphi_{m,i}}{\bfphi_{m',i'}}
+\ip{\bfe_i\otimes\bff_j}{\bfzero(\bfe_{i'}\otimes\bff_{j'})}
=\ip{\bfphi_{m,i}}{\bfphi_{m',i'}},
\end{equation*}
which is unimodular since $\set{\bfphi_{m,i}}_{m=1,}^{M}\,_{i=1}^{S+L}$ is an ETF of type $(K,L,S)$, and has been scaled so that $\norm{\bfphi_{m,i}}^2=S$ for all $m$ and $i$.
Next, if we instead have $u=u'$ and $m=m'$ then Lemma~\ref{lem.embed} gives $\bfE_{u,m}^*\bfE_{u,m}^{}=\bfI$ and so~\eqref{eq.proof of new ETF 5} becomes
\begin{equation}
\label{eq.proof of new ETF 6}
\ip{\bfpsi_{u,m,i,j}}{\bfpsi_{u,m,i',j'}}
=\norm{\bfdelta_u}^2\ip{\bfphi_{m,i}}{\bfphi_{m,i'}}+\ip{\bfe_i\otimes\bff_j}{\bfe_{i'}\otimes\bff_{j'}}
=\ip{\bfphi_{m,i}}{\bfphi_{m,i'}}+\ip{\bfe_i}{\bfe_{i'}}\ip{\bff_j}{\bff_{j'}}.
\end{equation}
In particular, when $u=u'$, $m=m'$ and $i\neq i'$, the fact that $\set{\bfe_i}_{i=1}^{S+L}$ is orthogonal implies that~\eqref{eq.proof of new ETF 6} reduces to $\ip{\bfphi_{m,i}}{\bfphi_{m,i'}}$, which is unimodular for the same reason as the previous case.
The final remaining case is the most interesting: when $u=u'$, $m=m'$, $i=i'$ but $j\neq j'$,
we have $0=\ip{1\oplus\bff_j}{1\oplus\bff_{j'}}=1+\ip{\bff_j}{\bff_{j'}}$ and so $\ip{\bff_j}{\bff_{j'}}=-1$;
when combined with the fact that $\norm{\bfphi_{m,i}}^2=S$ and $\norm{\bfe_i}^2=S+L$, this implies that in this case~\eqref{eq.proof of new ETF 6} becomes
\begin{equation*}
\ip{\bfpsi_{u,m,i,j}}{\bfpsi_{u,m,i,j'}}
=\norm{\bfphi_{m,i}}^2+\norm{\bfe_i}^2\ip{\bff_j}{\bff_{j'}}
=S+(S+L)(-1)
=-L,
\end{equation*}
where, in Definition~\ref{def.ETF types}, we have assumed that $L$ is either $1$ or $-1$.
\end{proof}

The construction of Theorem~\ref{thm.new ETF} leads to the concept of ETFs of type $(K,L,S)$.
To clarify, the construction~\eqref{eq.new ETF 3} and the above proof of its equiangularity is valid for any initial $\ETF(D,N)$ and any $K$-GDD of type $M^U$, provided $M=\frac{N}{S+L}$ for some $L\in\set{-1,1}$, and $S+L$ divides \smash{$R=\frac{M(U-1)}{K-1}$}.
However, it turns out that the first $UD$ rows of the corresponding synthesis operator have squared-norm $(W+1)\frac{NS}{D}$, whereas the last $B$ rows have squared-norm $(W+1)K(S+L)$.
As such, the equiangular vectors~\eqref{eq.new ETF 3} are only a tight frame when \smash{$K=\frac{NS}{D(S+L)}$}.
Applying the techniques of the proof of Theorem~\ref{thm.parameter types} then leads to the expressions for $(D,N)$ in terms of $(K,L,S)$ given in~\eqref{eq.ETF param in terms of type param}.
These facts are not explicitly discussed in our proof above since any equal-norm vectors that attain the Welch bound are automatically tight.

\begin{remark}
\label{rem.recursive}
There is no apparent value to recursively applying Theorem~\ref{thm.new ETF}.
To elaborate, given an $\ETF(D,N)$ of type $(K,L,S)$ and a $K$-GDD of type $M^U$ where $M=S(K-1)+L$ and $U$ satisfies~\eqref{eq.new ETF 2},
Theorem~\ref{thm.new ETF} yields an $\ETF(D',N')$ of type $(K,L,S')$ where $S'=\frac{MU-L}{K-1}$.
The ``$M$" parameter of this new ETF is thus $M'=S'(K-1)+L=(MU-L)+L=MU$,
and we can apply Theorem~\ref{thm.new ETF} a second time provided we have a $K$-GDD of type $(MU)^{U'}$ where $U'$ satisfies the appropriate analog of~\eqref{eq.new ETF 2}, namely
\begin{equation}
\label{eq.recursive 1}
\tfrac{(K-2)(U'-1)}{(S'+L)(K-1)}\in\bbZ.
\end{equation}
Doing so yields an ETF of type $(K,-1,S'')$ where $S''=\tfrac{M'U'-L}{K-1}=\tfrac{MUU'-L}{K-1}$.
However, under these hypotheses, there is a simpler way to construct an ETF of this same type.
Indeed, using the first GDD to fill the holes of the second GDD via Lemma~\ref{lem.filling holes} produces a $K$-GDD of type $M^{UU'}$.
Moreover, $UU'$ is a value of ``$U$" that satisfies~\eqref{eq.new ETF 2}:
since $S'=S+R$,
\begin{equation*}
\tfrac{S'+L}{S+L}=\tfrac{S'-S}{S+L}+1=\tfrac{R}{S+L}+1=W+1\in\bbZ,
\end{equation*}
and this together with~\eqref{eq.new ETF 2} and \eqref{eq.recursive 1} imply
\begin{equation*}
\tfrac{(K-2)(UU'-1)}{(S+L)(K-1)}
=\tfrac{(K-2)[U(U'-1)+(U-1)]}{(S+L)(K-1)}
=U\tfrac{(S'+L)}{(S+L)}\tfrac{(K-2)(U'-1)}{(S'+L)(K-1)}+\tfrac{(K-2)(U-1)}{(S+L)(K-1)}
\in\bbZ.
\end{equation*}
As such, we can combine our original ETF of type $(K,L,S)$ with our $K$-GDD of type $M^{UU'}$ via Theorem~\ref{thm.new ETF} to directly produce an ETF of type $(K,L,\tfrac{MUU'-L}{K-1})$.
\end{remark}

We also point out that,
in a manner analogous to how every McFarland difference set can be viewed as a degenerate instance of a Davis-Jedwab difference set~\cite{DavisJ97},
every Steiner ETF can be regarded as a degenerate case of the construction of Theorem~\ref{thm.new ETF}.
Here, in a manner consistent with~\eqref{eq.ETF param in terms of type param},
we regard $(D,N)=(0,1)$ as being type $(K,L,S)=(K,1,0)$ where $K\geq 1$ is arbitrary.
Under this convention, $M=S(K-1)+L=1$, \smash{$R=\frac{U-1}{K-1}$} and $W=R$, meaning we need a $K$-GDD of type $1^U$, namely a $\BIBD(V,K,1)$ where $V=U$.
When such a BIBD exists,
Theorem~\ref{thm.new ETF} suggests we let $\set{\bfphi_{1,1}}$ be some fictitious ETF for the nonexistent space $\bbF^0$,
scaled so that \smash{$\norm{\bfphi_{1,1}}^2=0$}.
It also suggests we let \smash{$\set{\bfE_v}_{v=1}^{V}=\set{\bfE_{u,m}}_{u=1,}^{U}\,_{m=1}^{1}$} be the embedding operators of the BIBD, let $\bfe_1=1$, and let $\set{1\oplus\bff_j}_{j=0}^{R}$ be the columns of a possibly-complex Hadamard matrix of size $R$.
Under these conventions,~\eqref{eq.new ETF 3} reduces to a collection of $V(R+1)$ vectors
\smash{$\set{\bfpsi_{v,j}}_{v=1,}^{V}\,_{j=0}^{R}
=\set{\bfpsi_{u,m,i,j}}_{u=1,}^{U}\,_{m=1,}^{1}\,_{i=1,}^{1}\,_{j=0}^{R}$},
where for any $v=u$ and $j$, the fact that $\bfphi_{1,1}$ lies in a ``zero-dimensional space" makes it reasonable to regard
\begin{equation*}
\bfpsi_{v,j}
=\bfpsi_{u,1,1,j}
=(\bfdelta_v\otimes\bfphi_{1,1})\oplus\bigparen{\bfE_{u,1}(\bfe_1\otimes\bff_j)}
=\bfE_{u,1}(1\otimes\bff_j)
=\bfE_{u}\bff_j.
\end{equation*}
Here,
Theorem~\ref{thm.new ETF} leads us to expect that $\set{\bfE_v\bff_j}_{v=1,}^{V}\,_{j=0}^{R}$ is an ETF of type $(K,1,S+R)=(K,1,R)$.

This is indeed the case:
as discussed in the previous section,
$\set{\bfE_v\bff_j}_{v=1,}^{V}\,_{j=0}^{R}$ is by definition a Steiner ETF for $\bbF^B$ where $B=\frac{VR}{K}$,
and moreover letting $(K,L,S)=(K,1,R)$ in~\eqref{eq.ETF param in terms of type param} gives:
\begin{align*}
\tfrac{S}{K}[S(K-1)+L]
&=\tfrac{R}{K}[R(K-1)+1]
=\tfrac{R}{K}V
=B,\\
(S+L)[S(K-1)+L]
&=(R+1)[R(K-1)+1]
=V(R+1).
\end{align*}
In particular, every Steiner ETF is a positive ETF.

As a degenerate case of Remark~\ref{rem.recursive},
we further have that when a Steiner ETF is regarded as being positive,
applying Theorem~\ref{thm.new ETF} to it yields an ETF whose parameters match those of another Steiner ETF.
Indeed, since a Steiner ETF arising from a $\BIBD(V,K,1)$ is type $(K,1,R)$ where $R=\frac{V-1}{K-1}$, we can only apply Theorem~\ref{thm.new ETF} to it whenever there exists a $K$-GDD of type $M^U$ where $M=R(K-1)+1=V$ and $U$ satisfies~\eqref{eq.new ETF 2}.
In this case, the resulting ETF is type $(K,1,\frac{VU-1}{K-1})$.
However, under these same hypotheses, we can more simply use the $\BIBD(V,K,1)$ to fill the holes of the $K$-GDD of type $V^U$ via Lemma~\ref{lem.filling holes} to obtain a $\BIBD(UV,K,1)$,
and its corresponding Steiner ETF is type $(K,1,\frac{VU-1}{K-1})$.

\section{Families of positive and negative equiangular tight frames}

In this section, we use Theorems~\ref{thm.parameter types} and~\ref{thm.new ETF} to better our understanding of positive and negative ETFs,
in particular proving the existence of the new ETFs given in Theorems~\ref{thm.new neg ETS with K=4,5} and~\ref{thm.new neg ETS with K>5}.
Recall that by~\eqref{eq.ETF param in terms of type param}, any $\ETF(D,N)$ of type $(K,L,S)$ has $D=\tfrac{S}{K}[S(K-1)+L]=S^2-\tfrac{S(S-L)}{K}$,
and so $K$ necessarily divides $S(S-L)$.
As we shall see, it is reasonable to conjecture that such ETFs exist whenever this necessary condition is satisfied.

\subsection{Positive equiangular tight frames}

In light of Definition~\ref{def.ETF types} and Theorem~\ref{thm.parameter types},
an $\ETF(D,N)$ with $1<D<N$  is positive if and only if there exists integers $K\geq 1$ and $S\geq 2$ such that
\begin{equation}
\label{eq.positive ETF param}
D=\tfrac{S}{K}[S(K-1)+1]=S^2-\tfrac{S(S-1)}{K},
\quad
N=(S+1)[S(K-1)+1],
\end{equation}
or equivalently that \smash{$S=\bigbracket{\tfrac{D(N-1)}{N-D}}^{\frac12}$} and $K=\frac{NS}{D(S+1)}$ are integers.

As discussed in the previous section, every regular simplex is a $1$-positive ETF and vice versa,
and moreover, every Steiner ETF arising from a $\BIBD(V,K,1)$ is type $(K,1,R)$ where $R=\frac{V-1}{K-1}$.
Here, the fact that $K$ necessarily divides $S(S-1)=R(R-1)$ is equivalent to having that $D=B$ is necessarily an integer.
\textit{Fisher's inequality} also states that a $\BIBD(V,K,1)$ can only exist when $K\leq R$.
These necessary conditions on the existence of a $\BIBD(V,K,1)$ are known to also be sufficient when $K=2,3,4,5$, and also asymptotically sufficient in general:
for any $K\geq 2$, there exists $V_0$ such that for all $V\geq V_0$ with the property that $R=\frac{V-1}{K-1}$ and $B=\frac{VR}{K}=\frac{V(V-1)}{K(K-1)}$,
a $\BIBD(V,K,1)$ exists~\cite{AbelG07}.
Moreover, explicit infinite families of such BIBDs are known,
including affine geometries, projective geometries, unitals and Denniston designs~\cite{FickusMT12},
and each thus gives rise to a corresponding infinite family of (positive) Steiner ETFs.

That said, not every positive ETF is a Steiner ETF.
In particular, for any prime power $Q$,
there is a phased BIBD ETF~\cite{FickusJMPW19} whose Naimark complements are $\ETF(D,N)$ where
\begin{equation*}
D=\tfrac{Q^3+1}{Q+1},
\quad
N=Q^3+1,
\quad
S=\bigparen{\tfrac{N-1}{\frac ND-1}}^{\frac12}=\bigparen{\tfrac{Q^3}{Q}}^{\frac12}=Q,
\quad
K=\tfrac{NS}{D(S+1)}=\tfrac{(Q+1)Q}{Q+1}=Q,
\end{equation*}
namely an ETF of type $(Q,1,Q)$.
These parameters match those of a Steiner ETF arising from a projective plane of order $Q-1$.
However, such an ETF can exist even when no such projective plane exists.
For example, such an ETF exists when $Q=7$ despite the fact that no projective plane of order $Q-1=6$ exists.

Other ETFs of type $(K,1,S)$ have $K>S$,
and are thus not Steiner ETFs since the underlying BIBD would necessarily violate Fisher's inequality.
To elaborate,
when $K$ is a prime power,
the requirement that $K$ divides $S(S-1)$ where $S$ and $S-1$ are relatively prime implies that either $K$ divides $S$ or $S-1$,
and in either case $K\leq S$.
However, when $K$ is not a prime power,
we can sometimes choose it to be a divisor of $S(S-1)$ that is larger than $S$.

For example, taking $K=S(S-1)$ in~\eqref{eq.positive ETF param} gives $D=S^2-1$ and $N=(S^2-1)^2=D^2$.
Such an ETF thus corresponds to a SIC-POVM in a space whose dimension is one less than a perfect square.
Such SIC-POVMs are known to exist when $S=2,\dotsc,7,18$, and are conjectured to exist for all $S$~\cite{FuchsHS17,GrasslS17}.
Similarly, taking \smash{$K=\binom{S}{2}$} in~\eqref{eq.positive ETF param} gives $D=S^2-2$ and \smash{$N=\frac12(S^2-2)(S^2-1)=\binom{D+1}{2}$}.
We refer to such $(D,N)$ as being of \textit{real maximal type} since real-valued examples of such ETFs meet the real-variable version of the Gerzon bound,
and are known to exist when $S=3,5$.
Remarkably, it is known that a real ETF of this type does not exist when $S=7$~\cite{FickusM16}.
We also caution that there is a single pair $(D,N)$ with $N=\binom{D+1}{2}$ that is neither positive or negative despite the fact that an $\ETF(D,N)$ exists, namely $(D,N)=(3,6)$.

We summarize these facts as follows:

\begin{theorem}
\label{thm.known positive ETFs}
An ETF of type $(K,1,S)$ exists whenever:
\begin{enumerate}
\renewcommand{\labelenumi}{(\alph{enumi})}
\item
$K=1$ and $S\geq 2$, (regular simplices);\smallskip
\item
$K\geq 2$ and there exists a $\BIBD(S(K-1)+1,K,1)$ (Steiner ETFs~\cite{FickusMT12}), including:\smallskip
\begin{enumerate}
\renewcommand{\labelenumii}{(\roman{enumii})}
\item
$K=2,3,4,5$ and $S\geq K$ has the property that $K$ divides $S(S-1)$;\smallskip
\item
$K=Q$ and $S=\frac{Q^J-1}{Q-1}$ where $Q$ is a prime power and $J\geq 2$ (affine geometries);
\item
$K=Q+1$ and $S=\frac{Q^J-1}{Q-1}$ where $Q$ is a prime power and $J\geq 2$ (projective geometries);\smallskip
\item
$K=Q+1$ and $S=Q^2$ where $Q$ is a prime power (unitals);\smallskip
\item
$K=2^{J_1}$ and $S=2^{J_2}+1$ where $2\leq J_1<J_2$ (Denniston designs);\smallskip
\item
$K\geq2$ and $S$ is sufficiently large and has the property that $K$ divides $S(S-1)$;\smallskip
\end{enumerate}
\item
$K=Q$ and $S=Q$ whenever $Q$ is a prime power~\cite{FickusJMPW19};\smallskip
\item
$K=S(S-1)$ where $S=2,\dotsc,7,18$ (SIC-POVMs~\cite{FuchsHS17,GrasslS17});\smallskip
\item
$K=\binom{S}{2}$ where $S=3,5$ (real maximal type~\cite{FickusM16}).
\end{enumerate}
\end{theorem}

Because so much is already known regarding the existence of positive ETFs,
we could not find any examples where Theorem~\ref{thm.new ETF} makes a verifiable contribution.
As we now discuss, much less is known about negative ETFs,
and this gives Theorem~\ref{thm.new ETF} an opportunity to be useful.

\subsection{Negative equiangular tight frames}

By Definition~\ref{def.ETF types} and Theorem~\ref{thm.parameter types},
an $\ETF(D,N)$ with $1<D<N$  is negative if and only if there exists integers $K\geq 1$ and $S\geq 2$ such that
\begin{equation}
\label{eq.negative ETF param}
D=\tfrac{S}{K}[S(K-1)-1]=S^2-\tfrac{S(S+1)}{K},
\quad
N=(S-1)[S(K-1)-1],
\end{equation}
or equivalently that \smash{$S=\bigbracket{\tfrac{D(N-1)}{N-D}}^{\frac12}$} and $K=\frac{NS}{D(S-1)}$ are integers.
Here, since $\frac{N}{D}>1$ and $\frac{S}{S-1}>1$, we actually necessarily have that $K\geq 2$.

When $K=2$, \eqref{eq.negative ETF param} becomes $(D,N)=(\frac12S(S-1),(S-1)^2)$ and such ETFs exist for any $S\geq 3$,
being the Naimark complements of ETFs of type $(2,1,S-1)$.

In the $K=3$ case, \eqref{eq.negative ETF param} becomes
\begin{equation*}
D
=\tfrac{S(2S-1)}{3}
=S^2-\tfrac{S(S+1)}{3},
\quad
N=(S-1)(2S-1).
\end{equation*}
For $D$ to be an integer, we necessarily have $S\equiv0,2\bmod 3$.
Moreover, for any $S\geq 2$ with $S\equiv 0,2\bmod 3$, an ETF of type $(3,-1,S)$ exists. Indeed, the recent paper~\cite{FickusJMP18} gives a way to modify the Steiner ETF arising from a $\BIBD(V,3,1)$ to yield a Tremain $\ETF(D,N)$ with
\begin{equation*}
D=\tfrac16(V+2)(V+3),\quad
N=\tfrac12(V+1)(V+2),
\end{equation*}
for any $V\geq 3$ with $V\equiv1,3\bmod 6$.
Such an ETF is type $(3,-1,S)$ with $S=\frac12(V+3)$.
We also note that for every $J\geq1$, there is a harmonic $\ETF(D,N)$ arising from a Spence difference set~\cite{Spence77} with $(D,N)=(\tfrac12 3^{J}(3^{J+1}+1),\tfrac12 3^{J+1}(3^{J+1}-1))$,
and such an ETF is type $(3,-1,S)$ where $S=\frac12(3^{J+1}+1)$.
It remains unclear whether any Spence ETFs are unitarily equivalent to special instances of Tremain ETFs.

In the $K=4$ case, for any positive integer $J$, Davis and Jedwab~\cite{DavisJ97} give a difference set whose harmonic ETF has parameters~\eqref{eq.Davis Jedwab parameters} and so is a type $(4,-1,S)$ ETF where $S=\frac13(2^{2J+1}+1)$.
Beyond these examples,
a few other infinite families of negative ETFs are known to exist.
In particular, in order for the expression for $D$ given in~\eqref{eq.negative ETF param} to be an integer, $K$ necessarily divides $S(S+1)$,
and so it is natural to consider the special cases where $S=K$ and $S=K-1$.

When $S=K-1$, \eqref{eq.negative ETF param} gives that ETFs of type $(K,-1,K-1)$ have
\begin{equation*}
D=(K-1)(K-2),
\quad
N=K(K-2)^2.
\end{equation*}
Remarkably, these are the same $(D,N)$ parameters as those of an ETF of type $(K-2,1,K-1)$.
In particular, every Steiner ETF arising from an affine plane of order $Q$ is both of (positive) type $(Q,1,Q+1)$ and (negative) type $(Q+2,-1,Q+1)$.
More generally, a Steiner ETF arising from a $\BIBD(V,K,1)$ has $(D,N)=(B,V(R+1))$ where $R=\frac{V-1}{K-1}$ and $B=\frac{VR}{K}$,
and so is only negative when
\begin{equation}
\label{eq.when Steiner ETF is neg}
\tfrac{NS}{D(S-1)}
=\tfrac{V(R+1)R}{B(R-1)}
=\tfrac{K(R+1)}{R-1}
\end{equation}
is an integer.
When $R$ is even, $R-1$ and $R+1$ are relatively prime,
and this can only occur when $R-1$ divides $K$.
Here, since Fisher's inequality gives $K\leq R$, this happens precisely when either $R=K=2$ or $R=K+1$,
namely when the underlying BIBD is a $\BIBD(3,2,1)$ or is an affine plane of odd order $K$, respectively.
Meanwhile, when $R$ is odd, $R-1$ and $R+1$ have exactly one prime factor in common, namely $2$, and~\eqref{eq.when Steiner ETF is neg} is an integer precisely when $\frac12(R-1)$ divides $K$.
Since $K\leq R$, this happens precisely when either $R=K=3$, $R=K+1$ or $R=2K+1$,
namely when the underlying BIBD is the projective plane of order $2$, an affine plane of even order, or is a $\BIBD(V,K,1)$ where $V=(2K+1)(K-1)+1=K(2K-1)$ for some $K\geq 2$, respectively.
With regard to the latter,
it seems to be an open question whether a $\BIBD(K(2K-1),K,1)$ exists for every $K\geq 2$,
though a Denniston design provides one whenever $K=2^J$ for some $J\geq1$,
and they are also known to exist when $K=3,5,6,7$~\cite{MathonR07}.
For such ETFs,~\eqref{eq.when Steiner ETF is neg} becomes $K+1$,
meaning they are type $(K',-1,2K'-1)$ where $K'=K+1$.

Meanwhile, when $S=K$, \eqref{eq.negative ETF param} gives that ETFs of type $(K,-1,K)$ have
\begin{equation}
\label{eq.S=K param}
D=K^2-K-1,
\quad
N=(K-1)(K^2-K-1).
\end{equation}
The recently-discovered hyperoval ETFs of~\cite{FickusMJ16} are instances of such ETFs whenever $K=2^J+1$ for some $J\geq 1$.
In the $K=4$ case,~\eqref{eq.S=K param} becomes $(D,N)=(11,33)$,
and this seems to be the smallest set of positive or negative parameters for which the existence of a corresponding ETF remains an open problem.

Apart from these examples, it seems that only a finite number of other negative ETFs are known to exist.
For example, since $K$ necessarily divides $S(S+1)$, it is natural to also consider the cases where $K=S(S+1)$ and $K=\binom{S+1}{2}$.
When $K=S(S+1)$,
\eqref{eq.negative ETF param} becomes $(D,N)=(S^2-1,(S^2-1)^2)$.
In particular, every positive SIC-POVM is also negative.
Similarly, when $K=\binom{S+1}{2}$, \eqref{eq.negative ETF param} gives $N=\binom{D+1}{2}$ where $D=S^2-2$, meaning every positive $(D,N)$ of real maximal type is also negative.

The only other example of a negative ETF that we found in the literature was an $\ETF(22,176)$, which has type $(10,-1,5)$, and arises from a particular SRG.
When searching tables of known ETFs such as~\cite{FickusM16},
it is helpful to note that most positive and negative ETFs have redundancy $\frac{N}{D}>2$,
with the only exceptions being $1$-positive ETFs (regular simplices), $2$-negative ETFs (Naimark complements of $2$-positive ETFs),
$\ETF(2,3)$ when regarded as type $(3,-1,2)$,
and $\ETF(5,10)$, which are type $(3,-1,3)$.
Indeed, when $K\geq2$, any $K$-positive ETF has $\tfrac{N}{D}=\frac{K(S+1)}{S}>K\geq2$.
Meanwhile, when $K\geq 3$, any ETF of type $(K,-1,S)$ only has $\frac{K(S-1)}{S}=\tfrac{N}{D}\leq 2$ when $S\leq\frac{K}{K-2}$.
Since $K$ necessarily divides $S(S+1)$ where $S\geq 2$,
such an ETF can only exist when $K=3$ and $S=2,3$.
We summarize these previously-known constructions of negative ETFs as follows:

\begin{theorem}
\label{thm.known negative ETFs}
An ETF of type $(K,-1,S)$ exists whenever:
\begin{enumerate}
\renewcommand{\labelenumi}{(\alph{enumi})}
\item
$K=2$ and $S\geq 3$ (Naimark complements of $2$-positive ETFs);\smallskip
\item
$K=3$ and $S\geq 2$ with $S\equiv 0,2\bmod 3$ (Tremain ETFs~\cite{FickusJMP18} and Spence harmonic ETFs~\cite{Spence77});\smallskip
\item
$K=4$ and $S=\frac13(2^{2J+1}+1)$ for some $J\geq1$ (Davis-Jedwab harmonic ETFs~\cite{DavisJ97});\smallskip
\item
$K=Q+2$ and $S=K-1$ where $Q$ is a prime power (Steiner ETFs from affine planes~\cite{FickusMT12});\smallskip
\item
$K=2^J+1$ and $S=2K-1$ where $J\geq 1$ (Steiner ETFs from Denniston designs~\cite{FickusMT12});\smallskip
\item
$K=2^J+1$ and $S=K$ where $J\geq 1$ (hyperoval ETFs~\cite{FickusMJ16});\smallskip
\item
$K=S(S+1)$ where $S=2,\dotsc,7,18$ (SIC-POVMs~\cite{FuchsHS17});\smallskip
\item
$K=\binom{S+1}{2}$ where $S=3,5$ (real maximal type~\cite{FickusM16});\smallskip
\item
$(K,S)=(4,7),(6,11),(7,13),(8,15),(10,5)$ (various other ETFs~\cite{FickusM16}).
\end{enumerate}
\end{theorem}

From this list, we see that for any $K\geq 5$,
the existing literature provides at most a finite number of $K$-negative ETFs.
Theorem~\ref{thm.new ETF}(a) implies that many more negative ETFs exist:
if an ETF of type $(K,-1,S)$ exists, then an ETF of type $(K,-1,\frac{MU+1}{K-1})$ exists for all sufficiently large $U$ that satisfy~\eqref{eq.new ETF 1} and~\eqref{eq.new ETF 2}.
Combining this fact with Theorem~\ref{thm.known negative ETFs} immediately gives:

\begin{corollary}
There exists an infinite number of $K$-negative ETFs whenever:
\begin{enumerate}
\renewcommand{\labelenumi}{(\alph{enumi})}
\item
$K=Q+2$ where $Q$ is a prime power;\smallskip
\item
$K=Q+1$ where $Q$ is an even prime power;\smallskip
\item
$K=2,8,12,20,30,42,56,342$.
\end{enumerate}
\end{corollary}

In particular, we now know that there are an infinite number of values of $K$ for which an infinite number of $K$-negative ETFs exist,
with $K=14$ being the smallest open case.
With these asymptotic existence results in hand,
we now focus on applying Theorem~\ref{thm.new ETF} with explicit GDDs.

For example, the ``Mercedes-Benz" regular simplex $\ETF(2,3)$ is type $(K,L,S)=(3,-1,2)$ and so has $M=S(K-1)+L=3$.
Since $S+L=1$ divides $K-2=1$,
Theorem~\ref{thm.new ETF} can be applied with any $3$-GDD of type $3^U$ so as to produce an ETF of type
$(K,L,\tfrac{MU-L}{K-1})=(3,-1,\tfrac12(3U+1))$,
and moreover that the necessary conditions~\eqref{eq.GDD necessary conditions} on the existence of such a GDD reduce to~\eqref{eq.new ETF 1}, namely to having $U\geq 3$, $\tfrac12(U-1)\in\bbZ$ and $\tfrac12U(U-1)\in\bbZ$.
In fact, such GDDs are known to exist whenever these necessary conditions are satisfied~\cite{Ge07},
namely when $U\geq 3$ is odd.
(This also follows from the fact that such GDDs are equivalent to the incidence structures obtained by removing a parallel class from a resolvable Steiner triple system.)
That is, writing $U=2J+1$ for some $J\geq 1$, we can apply Theorem~\ref{thm.new ETF} with an $\ETF(2,3)$ and a known $3$-GDD of type $3^{2J+1}$ to produce an ETF of type $(3,-1,\tfrac12(3U+1))=(3,-1,3J+2)$.

In summary, applying Theorem~\ref{thm.new ETF} to an ETF of type $(3,-1,2)$ produces ETFs of type $(3,-1,S)$ for any $S\equiv 2\bmod 3$,
and so recovers the parameters of ``half" of all possible $3$-negative ETFs,
cf.\ Theorem~\ref{thm.known negative ETFs} and~\cite{FickusJMP18},
including the parameters of all harmonic ETFs arising from Spence difference sets.
To instead recover some of the ETFs of type $(3,-1,S)$ with $S\equiv0\bmod 3$,
one may, for example, apply Theorem~\ref{thm.new ETF} to the well-known $\ETF(5,10)$, which is type $(3,-1,3)$.

In order to obtain ETFs with verifiably new parameters,
we turn our attention to applying Theorem~\ref{thm.new ETF} to known $K$-negative ETFs with $K\geq 4$.
Here, the limiting factor seems to be a lack of knowledge regarding uniform $K$-GDDs:
while the literature has much to say when~$K=3,4,5$~\cite{Ge07},
we are relegated to well-known simple constructions involving Lemmas~\ref{lem.Wilson} and~\ref{lem.filling holes} whenever $K>5$.
As such, we consider the $K=4,5$ cases separately from those with $K>5$:

\begin{proof}[Proof of Theorem~\ref{thm.new neg ETS with K=4,5}]
The $\ETF(6,16)$ is type $(4,-1,3)$,
and so we can apply Theorem~\ref{thm.new ETF} whenever there exists a $4$-GDD of type $8^U$ where $U$ satisfies~\eqref{eq.new ETF 2}.
By Theorem~\ref{thm.new ETF}, the known necessary conditions~\eqref{eq.GDD necessary conditions} on the existence of such GDDs reduce to~\eqref{eq.new ETF 1}:
\begin{equation*}
U\geq 4,
\quad
\tfrac{U-1}{3}\in\bbZ,
\quad
\tfrac{U(U-1)}{3}=\tfrac{4U(U-1)}{4(3)}\in\bbZ,
\end{equation*}
namely to having $U\geq 4$ and $U\equiv 1\bmod 3$.
These necessary conditions on the existence of $4$-GDDs of type $8^U$ are known to be sufficient~\cite{Ge07}.
Moreover, for any such $U$, we have~\eqref{eq.new ETF 2} is automatically satisfied since
\smash{$\tfrac{U-1}{3}=\frac{2(U-1)}{2(3)}\in\bbZ$}.
Altogether, for any $U\geq 4$ with $U\equiv 1\bmod 3$, we can apply Theorem~\ref{thm.new ETF} to the ETF of type $(4,-1,3)$ with a $4$-GDD of type $8^U$,
and doing so produces an ETF of type $(4,-1,\frac13(8U+1))$.
Here, letting $U=1$ recovers the parameters of the original ETF.
Overall, writing $U=3J+1$ for some $J\geq0$,
this means that an ETF of type $(4,-1,S)$ exists whenever $S=\frac13(8U+1)=\frac13[8(3J+1)+1]=8J+3$ for any $J\geq0$,
namely whenever $S\equiv 3\bmod 8$.
In particular, for any $J\geq 1$ we can take $U=4^{J-1}$ to obtain an ETF of type $(4,-1,S)$ where $S=\tfrac13(8U+1)=\tfrac13(2^{2J-1}+1)$.
This means that applying Theorem~\ref{thm.new ETF} to the ETF of type $(4,-1,3)$ recovers the parameters~\eqref{eq.Davis Jedwab parameters} of harmonic ETFs corresponding to Davis-Jedwab difference sets.

In light of Remark~\ref{rem.recursive},
applying Theorem~\ref{thm.new ETF} to any ETF of type $(4,-1,S)$ with $S\equiv 3\bmod 8$ simply recovers a subset of the ETF types obtained by applying Theorem~\ref{thm.new ETF} to the ETF of type $(4,-1,3)$.
As such, to obtain more $4$-negative ETFs via Theorem~\ref{thm.new ETF},
we need to apply it to initial ETFs that lie outside of this family.
By Theorem~\ref{thm.known negative ETFs}, the existing literature gives one such set of parameters, namely ETFs of type $(4,-1,7)$, which have $(D,N)=(35,120)$.
Here, $M=20$, and a $4$-GDD of type $20^U$ can only exist if $U$ satisfies~\eqref{eq.new ETF 1}:
\begin{equation*}
U\geq 4,
\quad
\tfrac{U-1}{3}\in\bbZ,
\quad
\tfrac{U(U-1)}{3}=\tfrac{4U(U-1)}{4(3)}\in\bbZ.
\end{equation*}
Moreover, these necessary conditions are known to be sufficient~\cite{Ge07},
meaning Theorem~\ref{thm.new ETF} can be applied whenever $U$ also satisfies~\eqref{eq.new ETF 2},
namely \smash{$\tfrac{U-1}{9}=\tfrac{2(U-1)}{6(3)}\in\bbZ$}.
Thus, for any $U\equiv 1\bmod 9$, Theorem~\ref{thm.new ETF} yields an ETF of type $(4,-1,\frac13(20U+1))$.
In particular, an ETF of type $(4,-1,S)$ exists for any $S\equiv 7\bmod 60$.
ETFs of type $(4,-1,S)$ with $S\equiv 67\bmod 120$ thus arise from both constructions;
in the statement of Theorem~\ref{thm.new neg ETS with K=4,5},
we elect to not remove the overlapping values of $S$ from either family so as to not emphasize one family over the other, and possibly make it easier for future researchers to identify potential patterns.

In a similar manner, as summarized in Theorem~\ref{thm.known negative ETFs},
the existing literature provides ETFs of type $(5,-1,S)$ for exactly three values of $S$,
namely $\ETF(12,45)$, $\ETF(19,76)$ and $\ETF(63,280)$ which have $S=4,5,9$ and so
$M=15,19,35$, respectively.
For these particular values of $M$, the corresponding necessary conditions~\eqref{eq.GDD necessary conditions} on the existence of $5$-GDDs of type $M^U$ are known to be sufficient~\cite{Ge07},
meaning we only need $U$ to satisfy~\eqref{eq.new ETF 1} and~\eqref{eq.new ETF 2} for the corresponding value of $S$, namely to satisfy $U\geq 5$, $\frac14(U-1)\in\bbZ$ and that
\begin{align*}
\tfrac{5U(U-1)}{5(4)}\in\bbZ,
\quad
\tfrac{3(U-1)}{3(4)}\in\bbZ,&\text{ when }S=4,\\
\tfrac{6U(U-1)}{5(4)}\in\bbZ,
\quad
\tfrac{3(U-1)}{4(4)}\in\bbZ,&\text{ when }S=5,\\
\tfrac{10U(U-1)}{5(4)}\in\bbZ,
\quad
\tfrac{3(U-1)}{8(4)}\in\bbZ,&\text{ when }S=9.
\end{align*}
An ETF of type $(5,-1,S)$ thus exists when $S=\frac14(15U+1)$ with $U\equiv 1\bmod 4$, or $S=\frac14(19U+1)$ with $U\equiv 1,65\bmod 80$, or $S=\frac14(35U+1)$ with $U\equiv 1\bmod 32$.
That is, an ETF of type $(5,-1,S)$ exists when $S\equiv 4\bmod 15$,
or $S\equiv 5,309\bmod 380$, or $S\equiv 9\bmod 280$.
\end{proof}

In certain special cases, these techniques yield real ETFs:

\begin{proof}[Proof of Theorem~\ref{thm.new real ETFs}]
The ETF~\eqref{eq.new ETF 3} constructed in Theorem~\ref{thm.new ETF} is clearly real when the initial $\ETF(D,N)$ $\set{\bfphi_n}_{n=1}^{N}$ and the Hadamard matrices of size $S+L$ and $W+1$ are real.

In particular, a real $\ETF(63,280)$ exists~\cite{FickusM16},
and such ETFs are type $(5,-1,9)$.
Since a real Hadamard matrix of size $S+L=8$ exists,
then letting $M=S(K-1)+L=35$,
Theorem~\ref{thm.new ETF} yields a real ETF of type $(5,-1,\frac14(35U+1))$
whenever there exists a $5$-GDD of type $35^U$ where $U$ satisfies~\eqref{eq.new ETF 2} and there exists a real Hadamard matrix of size
\begin{equation*}
H
=W+1
=\tfrac{R}{S+L}+1
=\tfrac{M(U-1)}{(S+L)(K-1)}+1
=\tfrac{35(U-1)}{8(4)}+1
=\tfrac{35U-3}{32}.
\end{equation*}
Here, we recall from the proof of Theorem~\ref{thm.new neg ETS with K=4,5} that such GDDs exist whenever $U\equiv1\bmod 32$,
namely whenever $H\equiv1\bmod 35$.
Altogether, since $\frac14(35U+1)=\frac14[35(\frac{32H+3}{35})+1]=8H+1$,
Theorem~\ref{thm.new ETF} yields a real ETF of type $(5,-1,8H+1)$ whenever there exists a Hadamard matrix of size $H$ when $H\equiv1\bmod 35$.
An infinite number of such Hadamard matrices exist:
since $17$ is relatively prime to $140$,
Dirichlet's theorem implies an infinite number of primes $Q\equiv 17\bmod 140$ exist,
and each has the property that $Q\equiv 1\bmod 4$,
meaning that Paley's construction yields a real Hadamard matrix of size
$2(Q+1)\equiv 36\bmod 280$, in particular of size $2(Q+1)\equiv 1\bmod 35$.

Similarly, a real $\ETF(7,28)$ exists~\cite{FickusM16},
is type $(6,-1,3)$, and there exists a real Hadamard matrix of size $S+L=2$.
Letting $M=14$,
Theorem~\ref{thm.new ETF} thus yields a real ETF of type $(6,-1,\frac15(14U+1))$ whenever there exists a $6$-GDD of type $14^U$ where $U$ satisfies~\eqref{eq.new ETF 2} and there exists a real Hadamard matrix of size $H=W+1=\frac15(7U-2)$.
Here, the necessary conditions~\eqref{eq.new ETF 1} and~\eqref{eq.new ETF 2} on the existence of such a GDD reduce to having $U\equiv 1,6\bmod 15$,
namely to having $H\equiv 1,8\bmod 21$.
Since $K=6$, these necessary conditions are not known to be sufficient.
Nevertheless, they are asymptotically sufficient:
there exists $U_0$ such that for all $U\geq U_0$ with $U\equiv 1,6\bmod 15$,
there exists a $6$-GDD of type $14^U$ where $U$ satisfies~\eqref{eq.new ETF 2}.
As such, there exists $H_0$ such that for all $H\geq H_0$ with $H\equiv 1,8\bmod 21$,
if there exists a real Hadamard matrix of size $H$, then there exists a real ETF of type $(6,-1,\tfrac15(14U+1))=(6,-1,2H+1)$.
As above, an infinite number of such Hadamard matrices exist:
since $\gcd(73,84)=1$, there are an infinite number of primes $Q\equiv 73\bmod 84$,
and each has the property that $Q\equiv 1\bmod 4$,
meaning Paley's construction yields a real Hadamard matrix of size $2(Q+1)\equiv 148\bmod 168$,
in particular of size $2(Q+1)\equiv 1\bmod 21$.

Applying these same techniques to real $\ETF(22,176)$ and $\ETF(23,276)$ yields the infinite families stated in (c) and (d) of the result.
\end{proof}

We have now seen that the construction of Theorem~\ref{thm.new ETF} recovers all Steiner ETFs as a degenerate case, recovers the parameters of ``half" of all Tremain ETFs including those of all harmonic ETFs arising from Spence difference sets,
and also recovers the parameters~\eqref{eq.Davis Jedwab parameters} of harmonic ETFs arising from the Davis-Jedwab difference sets.
This is analogous to how the approach of~\cite{DavisJ97} unifies McFarland, Spence and Davis-Jedwab difference sets with those with parameters~\eqref{eq.Davis Jedwab parameters}.
From this perspective, the value of the generalization of~\cite{DavisJ97} given in Theorem~\ref{thm.new ETF} is that it permits weaker conclusions to be drawn from weaker assumptions:
while~\cite{DavisJ97} forms new difference sets (i.e., new harmonic ETFs) by combining given difference sets with building sets formed from a collection of hyperplanes,
Theorem~\ref{thm.new ETF} forms new ETFs by combining given ETFs with GDDs.

In fact, a careful read of~\cite{DavisJ97} indicates that the building sets used there to produce difference sets with parameters~\eqref{eq.Davis Jedwab parameters} are related to $4$-GDDs of type \smash{$8^{4^{J-1}}$} obtained by recursively using $\TD(4,2^{2j+1})$ to fill the holes of $\TD(4,2^{2j+3})$ for every $j=1,\dotsc,J-1$ via Lemma~\ref{lem.filling holes}.
Alternatively, such GDDs can be constructed by using Lemma~\ref{lem.Wilson} to combine a $\TD(4,8)$ with a $\BIBD(4^{J-1},4,1)$ arising from an affine geometry.
We now generalize these approaches, using Lemmas~\ref{lem.Wilson} and~\ref{lem.filling holes} to produce the GDDs needed to apply Theorem~\ref{thm.new ETF} to several known $K$-negative ETFs with $K>5$:

\begin{theorem}
\label{thm.new neg ETS with K>5}
If an ETF of type $(K,-1,S)$ exists where $\frac{K-2}{S-1}\in\bbZ$,
and a $\TD(K,M)$ exists where $M=S(K-1)-1$,
then an ETF of type $(K,-1,\tfrac{MU+1}{K-1})$ exists when either:
\begin{enumerate}
\renewcommand{\labelenumi}{(\roman{enumi})}
\item $U=1$ or $U=K$;\smallskip
\item
a $\BIBD(U,K,1)$ exists;\smallskip
\item
$U=K^J$ for some $J\geq2$, provided a $\TD(K,MK^j)$ exists for all $j=1,\dotsc,J-1$.
\end{enumerate}
As a consequence, an ETF of type $(K,-1,S)$ exists when either:
\begin{enumerate}
\renewcommand{\labelenumi}{(\alph{enumi})}
\item
$K=6$, $S=\frac15(9U+1)$ where either $U=6^J$ for some $J\geq 0$ or a $\BIBD(U,6,1)$ exists;\smallskip
\item
$K=6$, $S=\frac15(24U+1)$ where either $U=6^J$ for some $J\geq 0$ or a $\BIBD(U,6,1)$ exists;\smallskip
\item
$K=7$, $S=\tfrac1{6}(35U+1)$ where either $U=7^J$ for some $J\geq 0$ or a $\BIBD(U,7,1)$ exists;\smallskip
\item
$K=10$, $S=\tfrac1{9}(80U+1)$ where either $U=10^J$ for some $J\geq 0$ or a $\BIBD(U,10,1)$ exists;\smallskip
\item
$K=12$, $S=\frac1{11}(32U+1)$ when either $U=12^J$ for some $J\geq 0$ or a $\BIBD(U,12,1)$ exists.
\end{enumerate}
\end{theorem}

\begin{proof}
Since \smash{$\frac{K-2}{S-1}\in\bbZ$} and $M=S(K-1)-1$,
any $K$-GDD of type $M^U$ can be combined with the given ETF of type $(K,-1,S)$ via Theorem~\ref{thm.new ETF} in order to construct an ETF of type \smash{$(K,-1,\frac{MU+1}{K-1})$}.
(Since $\frac{M+1}{K-1}=S$, such an ETF also exists when $U=1$, namely the given initial ETF.)
For instance, the given $\TD(K,M)$ is a $K$-GDD of type $M^K$, and so such an ETF exists when $U=K$.
Other examples of such GDDs can be constructed by combining the given $\TD(K,M)$ with any $\BIBD(U,K,1)$---a $K$-GDD of type $1^U$---via Lemma~\ref{lem.Wilson}.
In particular, when $K$ is a prime power, we can construct these BIBDs from affine geometries of order $K$ to produce examples of such GDDs with $U=K^J$ for any $J\geq 2$.
GDDs with these parameters also sometimes exist even when $K$ is not a prime power:
if a $\TD(K,MK^j)$ exists for all $j=1,\dotsc,J-1$,
then recursively using $\TD(K,MK^{j-1})$ to fill the holes of $\TD(K,MK^j)$ for all $j=1,\dotsc,J-1$ via Lemma~\ref{lem.Wilson} gives a $K$-GDD of type \smash{$M^{K^J}$}.

We now apply these ideas to some known $K$-negative ETFs with $K>5$,
organized according to the families of Theorem~\ref{thm.known negative ETFs}.
In particular, ETFs of type $(Q+2,-1,Q+1)$ exist whenever $Q$ is a prime power,
and $S-1=Q$ divides $K-2=Q$.
For such ETFs, $M=Q(Q+2)$ and a $\TD(Q+2,Q(Q+2))$ equates to $Q$ MOLS of size $Q(Q+2)$,
which is known to occur when $Q=2,3,4,5,8$~\cite{AbelCD07}.
(When $Q$ and $Q+2$ are both prime powers,
the standard method~\cite{MacNeish22} only produces $Q-1$ MOLS of size $Q(Q+2)$.)
Taking $Q=2,3$ recovers a subset of the ETFs produced in Theorem~\ref{thm.new neg ETS with K=4,5},
and so we focus on $Q=4,5,8$.
In particular,
for these values of $Q$,
the above methods yield an ETF of type \smash{$(Q+2,-1,\tfrac1{Q+1}[Q(Q+2)U+1])$}
when either $U=1$, $U=Q+2$, or a $\BIBD(U,Q+2,1)$ exists.
Moreover, even when $Q+2$ is not a prime power,
such an ETF with $U=(Q+2)^J$ exists,
provided a $\TD(Q+2,Q(Q+2)^{j+1})$ exists for all $j=1,\dotsc,J-1$.
When $Q=4,8$, this occurs for any $J\geq 2$,
since for any $j\geq 1$, the number of MOLS of size $(4)6^{j+1}=2^{j+3}3^{j+1}$ is at least $\min\set{2^{j+3},3^{j+1}}-1\geq 7$,
while the number of MOLS of size $(8)10^{j+1}=2^{j+4}5^{j+1}$ is at least $\min\set{2^{j+4},5^{j+1}}-1\geq 23$.

Other new explicit infinite families of negative ETFs arise from SIC-POVMs, which by Theorem~\ref{thm.known negative ETFs} are ETFs of type $(K,-1,S)$ where $K=S(S+1)$.
Such ETFs always have the property that $S-1$ divides $K-2=(S-1)(S+2)$,
meaning we can apply the above ideas whenever there exists a $\TD(K,M)$ where
$M=(S-1)(S+1)^2$.
Such TDs exist when $S=2,3$, being $\TD(6,9)$ and $\TD(12,32)$, respectively.
When $S=2$ in particular, the above methods yield an ETF of type $(6,-1,\frac15(9U+1))$ when either $U=1$, $U=6$, a $\BIBD(U,6,1)$ exists, or $U=6^J$ for some $J\geq 2$.
This final family arises from the fact that a $\TD(6,(9)6^j)$ exists for all $j\geq 1$:
there are at least $5$ MOLS of size $54=(9)6$,
at least $8$ MOLS of size $324=9(6^2)$ (since there are $8$ MOLS of size $9$, and at least $8$ MOLS of size $36$)~\cite{AbelCD07},
while for any $j\geq 3$, the number of MOLS of size $(9)6^j=2^j 3^{j+2}$ is at least $2^j-1\geq 7$.
Similarly,
taking $S=3$ yields an ETF of type $(12,-1,\frac1{11}(32U+1))$ when either $U=1$, $U=12$, a $\BIBD(U,12,1)$ exists, or $U=12^J$ for some $J\geq 2$:
the number of MOLS of size $32(12)=384$ and $32(12)^2=4608$ is at least $15$~\cite{AbelCD07}, while for $j\geq 3$, the number of MOLS of size $32(12)^j=2^{2j+5}3^j$ is at least $\min\set{2^{2j+5},3^j}-1\geq 26$.
\end{proof}

To be clear, we view the ETFs produced in Theorem~\ref{thm.new neg ETS with K>5} as a ``proof of concept," and believe that GDD experts will be able to find many more examples of new ETFs using Theorems~\ref{thm.new ETF} and~\ref{thm.known negative ETFs}.
For this reason, we have omitted technical cases where an ETF of type $(K,-1,S)$ and a $\TD(K,M)$ exist where $M=S(K-1)-1$, but $S-1$ does not divide $K-2$,
such as when $S=K=9$, and when $S=2K-1$ where $K=6,7,8,9,17,65537$.
In such cases, one can still combine the TD with a $\BIBD(U,K,1)$ via Lemma~\ref{lem.Wilson} to produce a $K$-GDD of type $M^U$,
but~\eqref{eq.new ETF 2} is not automatically satisfied.

We also point out that though Theorem~\ref{thm.new ETF} is a generalization of~\cite{DavisJ97},
the new ETFs we have found here are disjoint from those produced by another known generalization of this same work.
In particular, the parameters~\eqref{eq.Davis Jedwab parameters} of Davis-Jedwab difference sets can be regarded as the $Q=2$ case of the more general family:
\begin{equation}
\label{eq.Davis Jedwab Chen parameters}
D=Q^{2J-1}\paren{\tfrac{2Q^{2J}+Q-1}{Q+1}},
\quad
N=4Q^{2J}\paren{\tfrac{Q^{2J}-1}{Q^2-1}},
\end{equation}
where $J\geq 1$ and $Q\geq 2$ are integers.
In particular, in~\cite{Chen97},
Chen generalizes the theory of~\cite{DavisJ97} in a way that produces difference sets with parameters~\eqref{eq.Davis Jedwab Chen parameters} for any $J\geq 1$ and any $Q$ that is either a power of $3$ or any even power of an odd prime.
For this reason, difference sets with parameters of the form~\eqref{eq.Davis Jedwab Chen parameters} are said to be \textit{Davis-Jedwab-Chen difference sets}~\cite{JungnickelPS07}.
The inverse Welch bounds for the corresponding harmonic ETFs and their Naimark complements are always integers:
\begin{equation*}
S
=\bigbracket{\tfrac{D(N-1)}{N-D}}^{\frac12}
=\tfrac{2Q^{2J}+Q-1}{Q+1},
\quad
\bigbracket{\tfrac{(N-D)(N-1)}{D}}^{\frac12}
=\tfrac{2Q^{2J}-Q-1}{Q-1}.
\end{equation*}
However, such ETFs are seldom positive or negative.
Indeed, a direct computation reveals
\begin{equation*}
\tfrac{NS}{D(S+1)}
=2+\tfrac{2Q(Q^{2J-2}-1)}{(Q-1)(Q^{2J-1}+1)},
\quad
\tfrac{NS}{D(S-1)}
=2+\tfrac{2}{Q-1}.
\end{equation*}
Thus,
by Theorem~\ref{thm.parameter types},
the only such ETFs that are positive or negative are of type
$(2,1,2Q-1)$, $(3,-1,\tfrac12(9^J+1))$ or $(4,-1,\tfrac13(2^{2J+1}+1))$,
corresponding to the special cases of~\eqref{eq.Davis Jedwab Chen parameters} where $J=1$, $Q=3$ and $Q=2$, respectively.
That is, the only overlap between the $K$-negative ETFs constructed in Theorems~\ref{thm.new neg ETS with K=4,5} and~\ref{thm.new neg ETS with K>5} and the ETFs constructed in~\cite{Chen97} are those with parameters~\eqref{eq.Davis Jedwab Chen parameters}.

\section{Conclusions}

Theorems~\ref{thm.new neg ETS with K=4,5} and~\ref{thm.new neg ETS with K>5} make some incremental progress towards resolving the ETF existence problem.
By comparing the existing ETF literature, as summarized in~\cite{FickusM16} for example,
against Theorems~\ref{thm.known positive ETFs} and~\ref{thm.known negative ETFs},
we find that every known ETF is either an orthonormal basis,
or has the property that either it or its Naimark complement is a regular simplex, has $N=2D$ or $N=2D\pm1$, is a SIC-POVM, arises from a difference set, or is either positive or negative.
In particular, every known $\ETF(D,N)$ with $N>2D>2$ is either a SIC-POVM ($N=D^2$), or is a harmonic ETF, or has $N=2D+1$ where $D$ is odd, or is either positive or negative.
This fact, along with the known nonexistence of the $\ETF(3,8)$~\cite{Szollosi14},
and the available numerical evidence~\cite{TroppDHS05},
leads us to make the following conjecture:

\begin{conjecture}
\label{con.complex}
If $N>2D>2$, an $\ETF(D,N)$ exists if and only if $N=D^2$ or $\frac{D(D-1)}{N-1}\in\bbZ$ or $(D,N)$ is positive or negative (Definition~\ref{def.ETF types}).
\end{conjecture}

This is a substantial strengthening of an earlier conjecture made by one of the authors,
namely that if $N>D>1$ and an $\ETF(D,N)$ exists, then one of the three numbers $D$, $N-D$ and $N-1$ necessarily divides the product of the other two.
To help resolve Conjecture~\ref{con.complex}, it would in particular be good to know whether an $\ETF(9,25)$ exists: this is the smallest value of $D$ for which there exists an $N$ such that $N-1$ divides $D(D-1)$ but no known ETF exists.
For context, we note that there are numerous pairs $(D,N)$ for which $N-1$ divides $D(D-1)$ and an $\ETF(D,N)$ is known to exist, despite the fact that a difference set of that size does not exist~\cite{Gordon18}, including
\begin{gather*}
(35,120), (40,105), (45,100), (63,280), (70,231), (77,210), (91,196), (99,540), (130,560),\\
 (143,924), (176,561), (187,528), (208,1105), (231,484), (247,780), (260,741).
\end{gather*}

It would also be good to know whether an $\ETF(11,33)$ exists,
since such an ETF would be type $(4,-1,4)$:
as detailed in the previous section, an ETF of type $(K,L,S)$ can only exist when $K$ divides $S(S-L)$, and moreover this necessary condition for existence is known to be sufficient when $L=1$ and $K=1,2,3,4,5$, as well as when $L=-1$ and $K=2,3$.

The evidence also supports an analogous conjecture in the real case.
In fact, comparing the relevant literature~\cite{Brouwer17,FickusM16} against Theorems~\ref{thm.known positive ETFs} and~\ref{thm.known negative ETFs},
we find that every known real $\ETF(D,N)$ with $N>2D>2$ is either positive or negative.
Moreover, when $1<D<N-1$, $N\neq 2D$ and a real $\ETF(D,N)$ exists,
then it and its Naimark complements' Welch bounds are necessarily the reciprocals of odd integers~\cite{SustikTDH07}.
In particular, any such ETF automatically satisfies one of the two integrality conditions given in Theorem~\ref{thm.parameter types} that characterize positive and negative ETFs.
Conversely, since $S(K-1)+KL$ is odd whenever $S$ is odd,
Theorem~\ref{thm.parameter types} implies that any ETF of type $(K,L,S)$ satisfies the necessary conditions of~\cite{SustikTDH07} when $S$ is odd.
That said,
real $\ETF(19,76)$, $\ETF(20,96)$ and $\ETF(47,1128)$ do not exist~\cite{AzarijaM15,AzarijaM16,FickusM16},
despite the fact that $(19,76)$ is type $(5,-1,5)$,
$(20,96)$ is both type $(4,1,5)$ and type $(6,-1,5)$,
and $(47,1128)$ is both type $(21,1,7)$ and type $(28,-1,7)$.
These facts suggest the following analog of Conjecture~\ref{con.complex}:

\begin{conjecture}
If $N>2D>2$ and a real $\ETF(D,N)$ exists, then $(D,N)$ is positive or negative.
\end{conjecture}

\section*{Acknowledgments}
The views expressed in this article are those of the authors and do not reflect the official policy or position of the United States Air Force, Department of Defense, or the United States Government.
This work was partially supported by the Summer Faculty Fellowship Program of the United States Air Force Research Laboratory.


\begin{thebibliography}{WW}

\bibitem{AbelCD07}
R.~J.~R.~Abel, C.~J.~Colbourn, J.~H.~Dinitz,
Mutually Orthogonal Latin Squares (MOLS),
in: C.J.~Colbourn, J.H.~Dinitz (Eds.), Handbook of Combinatorial Designs, Second Edition (2007) 160--193.

\bibitem{AbelG07}
R.~J.~R.~Abel, M.~Greig,
BIBDs with small block size,
in: C.J.~Colbourn, J.H.~Dinitz (Eds.), Handbook of Combinatorial Designs, Second Edition (2007) 72--79.

\bibitem{AzarijaM15}
J.~Azarija, T.~Marc,
There is no (75,32,10,16) strongly regular graph,
arXiv:1509.05933.

\bibitem{AzarijaM16}
J.~Azarija, T.~Marc,
There is no (95,40,12,20) strongly regular graph,
arXiv:1603.02032.

\bibitem{BajwaCM12}
W.~U.~Bajwa, R.~Calderbank, D.~G.~Mixon,
Two are better than one: fundamental parameters of frame coherence,
Appl.\ Comput.\ Harmon.\ Anal.\ 33 (2012) 58-–78.

\bibitem{BandeiraFMW13}
A.\ S.\ Bandeira, M.\ Fickus, D.\ G.\ Mixon, P.\ Wong,
The road to deterministic matrices with the Restricted Isometry Property,
J.\ Fourier Anal.\ Appl.\ 19 (2013) 1123--1149.

\bibitem{BargGOY15}
A.~Barg, A.~Glazyrin, K.~A.~Okoudjou, W.-H.~Yu,
Finite two-distance tight frames,
Linear Algebra Appl.\ 475 (2015) 163--175.

\bibitem{BodmannE10}
B.~G.~Bodmann, H.~J.~Elwood,
Complex equiangular Parseval frames and Seidel matrices containing $p$th roots of unity,
Proc.\ Amer.\ Math.\ Soc.\ 138 (2010) 4387--4404.

\bibitem{BodmannPT09}
B.~G.~Bodmann, V.~I.~Paulsen, M.~Tomforde,
Equiangular tight frames from complex Seidel matrices containing cube roots of unity,
Linear Algebra Appl.\ 430 (2009) 396--417.

\bibitem{BrackenMW06}
C.~Bracken, G.~McGuire, H.~Ward,
New quasi-symmetric designs constructed using mutually orthogonal {L}atin squares and {H}adamard matrices,
Des.\ Codes Cryptogr.\ 41 (2006) 195--198.

\bibitem{Brouwer07}
A.~E.~Brouwer,
Strongly regular graphs,
in: C.~J.~Colbourn, J.~H.~Dinitz (Eds.), Handbook of Combinatorial Designs, Second Edition (2007) 852-–868.

\bibitem{Brouwer17}
A.~E.~Brouwer,
Parameters of Strongly Regular Graphs,
http://www.win.tue.nl/$\sim$aeb/graphs/srg/

\bibitem{Chang76}
K.~I.~Chang,
An existence theory for group divisible designs,
Ph.D.\ Thesis, The Ohio State University, 1976.

\bibitem{Chen97}
Y.~Q.~Chen,
On the existence of abelian Hadamard difference sets and a new family of difference sets,
Finite Fields Appl.\ 3 (1997) 234--256.

\bibitem{CorneilM91}
D.~Corneil, R. Mathon, eds.,
Geometry and combinatorics: Selected works of J.~J.~Seidel,
Academic Press, 1991.

\bibitem{CoutinkhoGSZ16}
G.~Coutinho, C.~Godsil, H.~Shirazi, H.~Zhan,
Equiangular lines and covers of the complete graph,
Linear Algebra Appl.\ 488 (2016) 264--283.

\bibitem{DavisJ97}
J.~A.~Davis, J.~Jedwab,
A unifying construction for difference sets,
J.\ Combin.\ Theory Ser.~A 80 (1997) 13--78.

\bibitem{DingF07}
C.~Ding, T.~Feng,
A generic construction of complex codebooks meeting the Welch bound,
IEEE Trans.\ Inform.\ Theory 53 (2007) 4245--4250.

\bibitem{FickusJMP18}
M.~Fickus, J.~Jasper, D.~G.~Mixon, J.~D.~Peterson,
Tremain equiangular tight frames,
J.\ Combin.\ Theory Ser.~A  153 (2018) 54–-66.

\bibitem{FickusJMP19}
M.~Fickus, J.~Jasper, D.~G.~Mixon, J.~D.~Peterson,
Hadamard equiangular tight frames,
submitted, arXiv:1703.05353.

\bibitem{FickusJMPW18}
M.~Fickus, J.~Jasper, D.~G.~Mixon, J.~D.~Peterson, C.~E.~Watson,
Equiangular tight frames with centroidal symmetry,
to appear in Appl.\ Comput.\ Harmon.\ Anal.\

\bibitem{FickusJMPW19}
M.~Fickus, J.~Jasper, D.~G.~Mixon, J.~D.~Peterson, C.~E.~Watson,
Polyphase equiangular tight frames and abelian generalized quadrangles,
to appear in Appl.\ Comput.\ Harmon.\ Anal.\

\bibitem{FickusM16}
M.~Fickus, D.~G.~Mixon,
Tables of the existence of equiangular tight frames,
arXiv:1504.00253 (2016).

\bibitem{FickusMJ16}
M.~Fickus, D.~G.~Mixon, J.~Jasper,
Equiangular tight frames from hyperovals,
IEEE Trans.\ Inform.\ Theory 62 (2016) 5225--5236.

\bibitem{FickusMT12}
M.~Fickus, D.~G.~Mixon, J.~C.~Tremain,
Steiner equiangular tight frames,
Linear Algebra Appl.\ 436 (2012) 1014--1027.

\bibitem{FuchsHS17}
C.~A.~Fuchs, M.~C.~Hoang, B.~C.~Stacey,
The SIC question: history and state of play,
Axioms 6 (2017) 21.

\bibitem{Ge07}
G.~Ge,
Group divisible designs,
in: C.~J.~Colbourn, J.~H.~Dinitz (Eds.), Handbook of Combinatorial Designs, Second Edition (2007) 255--260.

\bibitem{Godsil92}
C.~D.~Godsil,
Krein covers of complete graphs,
Australas.\ J.\ Combin.\ 6 (1992) 245--255.

\bibitem{GoethalsS70}
J.~M.~Goethals, J.~J.~Seidel,
Strongly regular graphs derived from combinatorial designs,
Can.\ J.\ Math.\ 22 (1970) 597--614.

\bibitem{Gordon18}
D.~Gordon,
La Jolla Covering Repository,
https://www.ccrwest.org/diffsets.html.

\bibitem{GrasslS17}
M.~Grassl, A.~J.~Scott,
Fibonacci-Lucas SIC-POVMs,
J.\ Math.\ Phys.\ 58 (2017) 122201.

\bibitem{HolmesP04}
R.~B.~Holmes, V.~I.~Paulsen,
Optimal frames for erasures,
Linear Algebra Appl.\ 377 (2004) 31--51.

\bibitem{IversonJM16}
J.\ W.\ Iverson, J.\ Jasper, D.\ G.\ Mixon,
Optimal line packings from nonabelian groups,
submitted, arXiv:1609.09836.

\bibitem{JasperMF14}
J.~Jasper, D.~G.~Mixon, M.~Fickus,
Kirkman equiangular tight frames and codes,
IEEE Trans.\ Inform.\ Theory 60 (2014) 170-–181.

\bibitem{JungnickelPS07}
D.~Jungnickel, A.~Pott, K.~W.~Smith,
Difference sets,
in: C.~J.~Colbourn, J.~H.~Dinitz (Eds.), Handbook of Combinatorial Designs, Second Edition (2007) 419--435.

\bibitem{vanLintS66}
J.~H.~van~Lint, J.~J.~Seidel,
Equilateral point sets in elliptic geometry,
Indag.\ Math.\ 28 (1966) 335--348.

\bibitem{LamkenW00}
E.~R.~Lamken, R.~M.~Wilson,
Decompositions of edge-colored complete graphs,
J.\ Combin.\ Theory Ser.~A  89 (2000) 149--200.

\bibitem{LemmensS73}
P.~W.~H.~Lemmens, J.~J.~Seidel,
Equiangular lines,
J.\ Algebra 24 (1973) 494--512.

\bibitem{MathonR07}
R.~Mathon, A.~Rosa,
$2-(v,k,\lambda)$ designs of small order,
in: C.~J.~Colbourn, J.~H.~Dinitz (Eds.), Handbook of Combinatorial Designs, Second Edition (2007) 25--58.

\bibitem{MacNeish22}
H.~F.~MacNeish,
Euler Squares,
Ann.\ of Math.\ 23 (1922) 221--227.

\bibitem{McFarland73}
R.~L.~McFarland,
A family of difference sets in non-cyclic groups,
J.\ Combin.\ Theory Ser.~A 15 (1973) 1--10.

\bibitem{McGuire97}
G.~McGuire,
Quasi-symmetric designs and codes meeting the Grey-Rankin bound,
J.\ Combin.\ Theory Ser.~A 78 (1997) 280--291.

\bibitem{Mohacsy11}
Hedvig Moh\'{a}csy,
The asymptotic existence of group divisible designs of large order with index one,
J.\ Combin.\ Theory Ser.~A  118 (2011) 1915--1924.

\bibitem{Renes07}
J.~M.~Renes,
Equiangular tight frames from Paley tournaments,
Linear Algebra Appl.~426 (2007) 497--501.

\bibitem{RenesBSC04}
J.~M.~Renes, R. Blume-Kohout, A.~J.~Scott, C.~M.~Caves,
Symmetric informationally complete quantum measurements,
J.\ Math.\ Phys.\ 45 (2004) 2171--2180.

\bibitem{Seidel76}
J.~J.~Seidel,
A survey of two-graphs,
Coll.\ Int.\ Teorie Combin., Atti dei Convegni Lincei 17, Roma (1976) 481--511.

\bibitem{Spence77}
E.~Spence,
A family of difference sets,
J.\ Combin.\ Theory Ser.~A 22 (1977) 103--106.

\bibitem{Strohmer08}
T.~Strohmer,
A note on equiangular tight frames,
Linear Algebra Appl.~429 (2008) 326–-330.

\bibitem{StrohmerH03}
T.~Strohmer, R.~W.~Heath,
Grassmannian frames with applications to coding and communication,
Appl.\ Comput.\ Harmon.\ Anal.\ 14 (2003) 257--275.

\bibitem{SustikTDH07}
M.~A.~Sustik, J.~A.~Tropp, I.~S.~Dhillon, R.~W.~Heath,
On the existence of equiangular tight frames,
Linear Algebra Appl.\ 426 (2007) 619--635.

\bibitem{Szollosi14}
F.\ Sz\"{o}ll\H{o}si,
All complex equiangular tight frames in dimension 3,
arXiv:1402.6429.

\bibitem{Tropp05}
J.~A.~Tropp,
Complex equiangular tight frames,
Proc.\ SPIE 5914 (2005) 591401/1--11.

\bibitem{TroppDHS05}
J.~A.~Tropp, I.~S.~Dhillon, R.~W.~Heath,~Jr., T.~Strohmer,
Designing structured tight frames via an alternating projection method,
IEEE Trans.\ Inform.\ Theory 51 (2005) 188--209.

\bibitem{Turyn65}
R.~J.~Turyn,
Character sums and difference sets,
Pacific J.\ Math.\ 15 (1965) 319--346.

\bibitem{Waldron09}
S.~Waldron,
On the construction of equiangular frames from graphs,
Linear Algebra Appl.\ 431 (2009) 2228--2242.

\bibitem{Welch74}
L.~R.~Welch,
Lower bounds on the maximum cross correlation of signals,
IEEE Trans.\ Inform.\ Theory 20 (1974) 397-–399.

\bibitem{Wilson72}
R.~M.~Wilson,
An existence theory for pairwise balanced designs I.\ Composition theorems and morphisms,
J.\ Combin.\ Theory Ser.~A 13 (1972) 220--245.

\bibitem{XiaZG05}
P.~Xia, S.~Zhou, G.~B.~Giannakis,
Achieving the Welch bound with difference sets,
IEEE Trans.\ Inform.\ Theory 51 (2005) 1900--1907.

\bibitem{Zauner99}
G.~Zauner,
Quantum designs: Foundations of a noncommutative design theory,
Ph.D.\ Thesis, University of Vienna, 1999.

\end{thebibliography}
\end{document}